\documentclass[12pt]{amsart}
\usepackage[english]{babel}

\usepackage{pifont}
\usepackage{todonotes}
\usepackage{comment}

\usepackage{geometry}
\usepackage{amsthm,amsmath,amssymb,amsfonts,amscd,mathtools}
\usepackage[shortlabels]{enumitem}
\usepackage{hyperref,cleveref}
\usepackage{multirow,float}

\theoremstyle{plain}
\newtheorem*{maintheorem}{Main Theorem}
\newtheorem{theorem}{Theorem}[section]

\newtheorem{lemma}[theorem]{Lemma}
\newtheorem{proposition}[theorem]{Proposition}

\newtheorem{corollary}[theorem]{Corollary}

\theoremstyle{definition}
\newtheorem{definition}[theorem]{Definition}
\newtheorem{remark}[theorem]{Remark}

\theoremstyle{remark}
\newtheorem{example}[theorem]{Example}

\numberwithin{equation}{section}
\numberwithin{figure}{section}
\numberwithin{table}{figure}

\newcommand{\pt}[1]{\left({#1}\right)}
\newcommand{\pq}[1]{\left[{#1}\right]}

\newcommand{\rest}[2]{\left.{#1}\right|_{#2}}
\newcommand{\pg}[1]{\left\{{#1}\right\}}
\newcommand{\abs}[1]{\left|{#1}\right|}
\newcommand\norm[1]{\left\lVert#1\right\rVert}

\newcommand{\vol}{\operatorname{Vol}}

\newcommand{\g}{\mathfrak{g}}
\newcommand{\ab}{\mathfrak{a}}
\newcommand{\abj}{\mathfrak{a}_J}

\newcommand{\da}{d_A}

\newcommand{\Z}{\mathbb{Z}}

\newcommand{\R}{\mathbb{R}}

\newcommand{\C}{\mathbb{C}}

\DeclareMathOperator{\real}{Re}

\DeclareMathOperator{\tr}{tr}
\DeclareMathOperator{\id}{Id}

\DeclareMathOperator{\ad}{ad}
\DeclareMathOperator{\End}{End}

\title[Special structures on almost abelian solvmanifolds]{Special structures on almost abelian solvmanifolds}

\author{Asia Mainenti}
\address[A. Mainenti]{Institute of Mathematics “Simion Stoilow” of the Romanian Academy, 21 Calea Grivitei, 010702 Bucharest, Romania}
\email{asia.mainenti@imar.ro}

\author{Andrei Moroianu}
\address[A. Moroianu]{Université Paris-Saclay, CNRS,  Laboratoire de mathématiques d'Orsay, 91405, Orsay, France, 
and Institute of Mathematics “Simion Stoilow” of the Romanian Academy, 21 Calea Grivitei, 010702 Bucharest, Romania}
\email{andrei.moroianu@math.cnrs.fr}

\thanks{}

\begin{document}

\begin{abstract} We characterize every almost abelian Lie algebra endowed with an integrable complex structure by a triple, called presentation, consisting of a real number, an element in some vector space and an endomorphism of that vector space. We then classify in terms of presentations the almost abelian Lie algebras admitting $p$-Kähler or $p$-pluriclosed structures, and in particular those carrying Kähler, balanced, pluriclosed and Gauduchon metrics.
Furthermore, we characterize the existence of LCK and Bismut-Ricci flat metrics in this setting.
\end{abstract}

\maketitle

\section{Introduction}

Over the last few decades, there has been increasing interest in structures on complex manifolds that generalize Kähler metrics.
Among these, we find the so-called \textit{special} Hermitian metrics, defined by imposing weaker cohomological assumptions than the Kähler one.
More precisely, let $(M,J)$ be a complex manifold of (complex) dimension $n$ and consider the $J$-twisted differential operator $d^c:=JdJ^{-1}$.
A first generalization was introduced by Gauduchon in \cite{Gaud77}: a Hermitian metric $\omega$ is called \textit{Gauduchon} if $dd^c\omega^{n-1}=0$.
In the same paper, Gauduchon proved that on a compact complex manifold of dimension at least $2$, there exists a unique, up to scalar rescaling, Gauduchon metric in every conformal class.
A more restrictive condition is obtained requiring that $d\omega^{n-1}=0$, in which case the Hermitian metric is called \textit{balanced}. 
This notion was thoroughly studied in \cite{michelsohn}, where Michelsohn established an interpretation in terms of the torsion of the Chern connection, and showed that the existence of balanced metrics can be characterized in terms of currents (just as in the Kähler case). 
A notable consequence is that the class of balanced manifolds is stable under modifications \cite{AlBas95}.

On the other hand, imposing weaker cohomological assumptions on the metric itself, we find \textit{pluriclosed}, also called \textit{strong Kähler with torsion (SKT)}, metrics, defined by the condition $dd^c\omega=0$, see \cite{Bis89}.
This notion is related to the Bismut connection associated to $\omega$, namely it is equivalent to the Bismut torsion being $d$-closed. Pluriclosed metrics are also particularly relevant in generalized complex geometry, due to their relation with generalized Kähler structures.
A further generalization, particularly relevant in conformal geometry, is given by \textit{locally conformally Kähler}, or \textit{LCK} metrics. 
These are defined by the relation $d\omega=\theta\wedge\omega$, for some closed $1$-form $\theta$.
For more details on LCK manifolds, we refer to \cite{ov24}.
Besides the restrictions given by cohomological or conformal assumptions, one can consider curvature constraints.
Among this wide class, we will turn our attention to \textit{Bismut-Ricci flat}, or \textit{Calabi-Yau with torsion (CYT)} metrics, namely those Hermitian metrics whose Bismut connection has vanishing Ricci curvature, see for instance \cite{FG24}.

All the aforementioned conditions are imposed on forms of bidegree either $(1,1)$ or $(n-1,n-1)$ which are powers of Hermitian metrics.
A natural question is then if analogous notions for $p$-th powers of Hermitian metrics, for $2\le p\le n-2,$ carry similar remarkable properties.
As it turns out, this is the case for the weaker cohomological assumption of being $dd^c$-closed, for instance for $p=n-2$: \textit{astheno-Kähler} metrics are defined by $dd^c\omega^{n-2}=0$, in \cite{JY93}, and have been used to obtain existence of Hermitian harmonic maps, with applications to the study of the fundamental group of the target.

On the other hand, for every $p\le n-2$, if $\omega^p$ is $d$-closed, then $\omega$ is $d$-closed itself, and so it is a Kähler metric.
With the goal of weakening the assumption, while maintaining the same cohomological constraint, the condition of being the power of a Hermitian metric can be interpreted as a positivity notion, and hence be replaced with a weaker one.
This is the motivation behind the definition of a \textit{$p$-Kähler} structure \cite{AA}, namely a $d$-closed $(p,p)$-form which is \textit{transverse}, in the sense that its restriction to every tangent subbundle of (complex) dimension $p$ is a positive volume form.
We note that a $1$-Kähler structure is a Kähler metric, whereas a $(n-1)$-Kähler structure is the exterior power of a balanced metric \cite{michelsohn}.
Furthermore, $p$-Kähler manifolds share many common behaviors with Kähler and balanced ones: they are relevant in the context of modifications of compact Kähler manifolds, and for $p=n-2$, they have applications in non-abelian Hodge theory, as noted in \cite{FM26}.
One of the main open problems in this setting is the Alessandrini-Bassanelli conjecture, stating that the existence of a $p$-Kähler structure implies the existence of a  $(p+1)$-Kähler one, for every $p\le n-2$.
Besides the evidence coming from the known examples, there are a few cases where the conjecture was proved to hold, such as a class of holomorphically parallelizable nilmanifolds \cite{LoG2025}, and nilmanifolds of complex dimension $4$ (cf. \cite{fm}).

In a similar fashion, a \textit{$p$-pluriclosed structure} on a complex manifold is  a $dd^c$-closed and transverse $(p,p)$-form \cite{Ale11}.
Once again, we recover pluriclosed metrics, for $p=1$, and powers of Gauduchon metrics, for $p=n-1$.
Moreover, $(n-2)$-th powers of astheno-Kähler metrics are $(n-2)$-pluriclosed structures.

\smallskip
The aim of the paper is to discuss the existence of the aforementioned structures, namely $p$-Kähler and $p$-pluriclosed structures, for $1\le p\le n-1$, and LCK and Bismut-Ricci flat metrics, on almost abelian solvmanifolds of real dimension $2n$.
In this setting, some particular cases were already studied in literature: Kähler metrics in \cite{LR}, balanced ones in \cite{fp23}, whereas pluriclosed metrics were discussed in \cite{AL}, LCK metrics in \cite{AO18}, Bismut-Ricci flat metrics in \cite{AL,fp23}, generalized Kähler structures in \cite{fp21}, pseudo-Kähler structures in \cite{CGG26}.
Observe that by symmetrization, the problem is equivalent to the analogous one at the Lie algebra level \cite{FG,fm}. If $\g$ is a Lie algebra with complex structure $J$, we say that $(\g,J)$ is Kähler, balanced, $p$-Kähler, etc. if there exists a Kähler, balanced, $p$-Kähler, etc. structure on $\g$ compatible with $J$.

We will thus consider almost abelian Lie algebras, namely Lie algebras that contain a codimension $1$ abelian ideal, endowed with an integrable complex structure.
It turns out that these can be described by triples $(\lambda,v,A)$, where $\lambda$ is a real number, $v$ is an element of a $2(n-1)$-dimensional vector space $\abj$ with complex structure $J$, and $A$ is an endomorphism of $\abj$ commuting with $J$. 
Such a triple, called a presentation, is determined up to the equivalence relation described in \eqref{changeTilde}. 

The main results that we obtain in the paper can be summed up as follows. 
\begin{maintheorem}
    Let $J$ be a complex structure on an almost abelian Lie algebra $\g$ of dimension $2n$ with presentation $(\lambda,v,A)$. 
Then,
\begin{enumerate}[label=(\roman*)]
    \item\label{TI1} $(\g,J)$ is Kähler if and only if $A$ is diagonalizable over $\C$, has purely imaginary spectrum, and $v$ belongs to the image of $A-\lambda \id $.
    \item\label{TI2} $(\g,J)$ is balanced if and only if $\tr(A)=0$ and $v$ belongs to the image of $A-\lambda \id $.
    \item\label{TI3} $(\g,J)$ is $p$-Kähler, for $p\le n-2$ if and only if it is Kähler.
    \item\label{TI4} $(\g,J)$ is pluriclosed if and only if $A$ is diagonalizable over $\C$ and its spectrum is contained in $(i\R)\cup (-\lambda/2+i\R)$.
    \item\label{TI5} $(\g,J)$ is Gauduchon if and only if  $\tr(A)\in\{0,-\lambda\}$, in which case every Hermitian metric is Gauduchon.
    \item\label{TI6} $(\g,J)$ is $p$-pluriclosed, for $p\le n-2$, if and only if $A$ is diagonalizable over $\C$ and the real parts of its eigenvalues satisfy a set of quadratic relations determined by $p$ and $\lambda$.
    \item\label{TI7}   $(\g,J)$ is LCK if and only if
        either $v\in (A-\lambda \id )(\abj)$ and $A$ is diagonalizable over $\C$, with all eigenvalues having the same real part, 
        or $n=2$ and $A=0$.
    \item\label{TI8}  $(\g,J)$ is Bismut-Ricci flat if and only if  one of the following holds:
    \begin{itemize}
        \item\label{iBRf1}  $\lambda\pt{2\lambda-\tr A}=0$ and $v$ is in the image of $(A-\lambda \id)$, or
        \item\label{iBRf2}  $\lambda\pt{2\lambda-\tr A}<0$ and  $v\in(A-\lambda \id)(\abj)+(\abj\setminus A(\abj))$.
    \end{itemize}
\end{enumerate}
\end{maintheorem}
In particular, \cref{TI3} confirms the Alessandrini-Bassanelli conjecture on almost abelian solvmanifolds, and the last item provides examples of unimodular almost abelian Lie algebras admitting $p$-pluriclosed structures only for specific values of $p$.
This shows that the analogue of the Alessandrini-Bassanelli conjecture for $p$-pluriclosed structures does not hold.
It is also worth noting that the last condition in \cref{TI1,TI2} is automatically satisfied if $\g$ is not unimodular, whereas it only depends on $A$ in the unimodular case, as $\lambda$ is forced to vanish by the further restrictions on the spectrum of $A$. 

Note that the Main Theorem recovers and generalizes the previously known results from the literature mentioned above concerning K\"ahler, balanced, pluriclosed, LCK and Bismut-Ricci flat metrics on almost abelian Lie algebras.

As an application, we confirm, in the almost abelian setting, some conjectures about coexistence of different types of special metrics.
The Fino-Vezzoni conjecture \cite{fv15} states that if a compact complex manifold admits both a balanced metric and a pluriclosed one, then it should also admit a K\"ahler metric. 
This was confirmed in  many cases, especially on special classes of compact quotients of Lie groups by lattices, including almost abelian solvmanifolds (in  \cite{fp23}).
We recover this last result, as a consequence of the Main Theorem.
This conjecture was extended in \cite{ov25}, to state that, in complex dimension higher than $2$, the existence of two types of metrics, among balanced, pluriclosed and LCK, should imply the existence of a Kähler metric. 
Comparing items \ref{TI2}, \ref{TI4} and \ref{TI7} in the Main Theorem, one can see that this more general conjecture is satisfied as well on unimodular almost abelian Lie algebras.
This proves that, on an almost abelian solvmanifold with complex structure, the existence of both a balanced (or pluriclosed) metric, and a left-invariant LCK one, implies the existence of a K\"ahler metric.

Furthermore, $\g$ is nilpotent if and only if for every presentation $(\lambda,v,A)$ of $(\g,J)$, one has $\lambda=0$ and $A$ nilpotent.
Therefore, we can describe explicitly the nilpotent cases among the ones appearing in the Main Theorem:
\begin{itemize}
    \item in cases \ref{TI1} and \ref{TI3}, $\g$ is abelian (recovering the well known fact that K\"ahler nilpotent Lie algebras are abelian);
    \item in case \ref{TI2},  $(\g,J)$ is balanced if and only if $v$ is in the image of $A$, and thus there are no restrictions on the nilpotency step of $\g$;
    \item in cases \ref{TI4} and \ref{TI6}, $A=0$, so in particular $\g$ is $2$-step nilpotent, and moreover, if $\g$ is not abelian, it is a direct product $\g\simeq\mathfrak{h}_3\oplus\R^{2n-3}$, where $\mathfrak h_3$ is the Heisenberg Lie algebra;
    \item every nilpotent Lie algebra satisfies the conditions in case \ref{TI5}, consistently with the fact that on a unimodular Lie algebra, every Hermitian metric is Gauduchon.
    \item in case \ref{TI7} the only possibility is $A=0$, and moreover if $n\ge 3$ one also has $v=0$, so $\g$ is abelian. Consequently, the only almost abelian nilpotent Lie algebra carrying LCK structures is $\mathfrak h_3 \times\R$ (which confirms \cite[Theorem 3.3]{AO18})."
\end{itemize}
Recall also that by \cite{AABRW}, if a nilpotent almost abelian Lie algebra admits a complex structure, then this is unique.

We will now describe the structure of the paper.

In \Cref{secAA}, we establish the correspondence between an almost abelian Lie algebra and a presentation $(\lambda,v,A)$, and  compute explicitly the action of the operators $d$ and $dd^c$ in terms of a given presentation.

In \Cref{secpK} we study $p$-Kähler structures, proving \cref{TI1,TI2,TI3} in the Main Theorem.
More precisely, \cref{TI1} is proved in Proposition \ref{lemma17}, \cref{TI2} in Proposition \ref{37} and \cref{TI3} in Theorem \ref{pKahler}.
The proof of the latter is based on several technical lemmas using the Jordan form of $A$, together with dimensionality reduction arguments relying on the properties of transverse $(p,p)$-forms.

In \Cref{secpPl} we consider $p$-pluriclosed structures. 
\Cref{TI4} of the Main Theorem is treated in  Proposition \ref{propSKT}, \cref{TI5} in Proposition \ref{410} and \cref{TI6} in Theorem \ref{AdiagPL}.
Moreover, we analyze the existence of lattices on the simply connected Lie groups with unimodular Lie algebras occurring in \cref{TI6} of the Main Theorem.

In \Cref{secLCK}, we conclude the proof of the Main Theorem.
We consider LCK metrics in Proposition \ref{p53}, proving \cref{TI7}, and Bismut-Ricci flat metrics in Proposition \ref{pCYT},  proving \cref{TI8}.

Finally, we collect in the Appendix some technical facts of linear algebra and combinatorics which are needed in the proof of the main results.

\smallskip

{\bf Acknowledgments.} 
The authors would like to thank Diego Conti for the careful reading of a preliminary version and useful comments, Jorge Lauret for suggesting to consider LCK and CYT metrics, and Sönke Rollenske for pointing out the connection with other references treating the nilpotent setting.
This work was partly supported by the PNRR-III-C9-2023-I8 grant CF 149/31.07.2023 {\it Conformal Aspects of Geometry and Dynamics}.

\section{Preliminaries on almost abelian Lie algebras}\label{secAA}

\subsection{Presentations of almost abelian Lie algebras}
A real Lie algebra $\g$ is called almost abelian if it has an abelian ideal $\ab$ of codimension $1$. Every almost abelian Lie algebra $(\g,\ab)$ is $2$-step solvable. Indeed, the derived Lie algebra $\g'=[\g,\g]$ is contained in $\ab$, which is abelian, so $[\g',\g']=0$.

Since $\ab$ is an ideal, every $x\in\g\setminus \ab$ determines an endomorphism $\varphi:=\rest{\ad_{x}}\ab\in\End(\ab)$. Then, $\g$ is isomorphic to the semidirect product
\begin{equation}\label{semiDirProd}
\g\simeq\ab\rtimes_\varphi \R.
\end{equation}
We will call the pair $(\ab,\varphi)$ a \textit{presentation} of $\g$. If $\tilde x$ is another vector in $\g\setminus \ab$, the corresponding endomorphism $\tilde\varphi$ is just a non-zero multiple of $\varphi$.

Conversely, by \cite[Proposition 1]{freibert}, if $(\ab,\varphi)$ and $(\ab_1,\varphi_1)$ are presentations of $\g$ and $\g_1$, then $\g\simeq\g_1$ if and only if $\varphi$ is conjugate to a non-zero multiple of $\varphi_1$ by some isomorphism $\ab\to\ab_1$.

\subsection{Presentations of almost abelian Lie algebras with complex structures}

Assume now that $(\g,\ab)$ is endowed with an integrable complex structure $J$, and denote $\ab_J:=\ab\cap J\ab$.
From now on, any complex structure will be assumed to be integrable, even when not specified.

\begin{remark}\label{r11}
    By \cite[Section 6]{LR}, it is known that the integrability of $J$ is equivalent to $\abj$ being an ideal, and the restriction of every adjoint map to $\abj$ commuting with the restriction of $J$ to $\abj$.
    Furthermore, the existence of complex structures was studied, in terms of the admissible Jordan blocks of the presentation $\varphi$ of $(\g,\ab)$, in \cite{ABDGH}.
\end{remark}

Every vector $y\in\ab\setminus\abj$ determines a unique vector $x:=-Jy\in\g\setminus\ab$, such that $Jx=y$. The vector $x$ defines a presentation $(\ab,\varphi:=\ad_x)$ of $\g$ as above. By Remark \ref{r11}, the restriction $A:=\varphi|_{\abj}$ is an endomorphism of $\abj$ commuting with $J|_{\abj}$. 
Furthermore, we can write
\begin{equation}\label{lv}
\varphi(y)=[x,y]=\lambda y+v
\end{equation}
for $\lambda\in\R$, $v\in\abj$. The triple $(\lambda,v,A)$ is called the \textit{presentation} of $(\g,\ab,J)$ determined by $y$. 
Conversely, for every tuple
\begin{equation*}
    \pt{\abj,J,\lambda,v,A}
\end{equation*}
consisting of a real, even dimensional vector space $\abj$ endowed with a complex structure $J$, a real number $\lambda\in\R$, a vector $v\in\abj$, and an endomorphism $A\in\End(\abj)$ commuting with $J$, we can construct an almost abelian Lie algebra $(\g,\ab)$ with a complex structure with presentation $(\lambda,v,A)$.

Indeed, we first define the abelian Lie algebra $\ab$ as the direct product $\ab:=\R\oplus\abj$, and denote by $y$ the generator of the factor $\R$. 
Then, we consider the endomorphism $\varphi$ of $\ab$ such that $\varphi|_{\abj}= A$ and $\varphi(y)=\lambda y+v$. Finally, we define the almost abelian Lie algebra $\g$ as the semidirect product $\g:=\ab\rtimes_\varphi \R$, and extend $J$ from $\abj$ to $\g$ by $Jx=y$, where $x$ is the generator of the factor $\R$ in $\g$ (so in particular $\ad_x=\varphi$). By the above discussion, $(\g,\ab,J)$ is an almost abelian Lie algebra with complex structure, whose presentation with respect to $y$ is $(\lambda,v,A)$.



We will now describe how the presentation changes with the choice of $y$, see also \cite{fp23}.
Let $\tilde y=cy+a$, for some non-zero $c\in\R^*$, $a\in\abj$. 
Then, $\tilde x=-J\tilde y=cx-Ja$ and the corresponding tuple $(\tilde\lambda,\tilde v,\tilde A)$ can be
computed as follows:
\begin{equation*}
\begin{aligned}
	\tilde A&=\rest{\ad_{\tilde x}}\abj=\rest{\ad_{cx-Ja}}\abj=c\rest{\ad_{x}}\abj=cA,\\
	[\tilde x,\tilde y]&=\tilde\lambda\tilde y+\tilde v\\
	 &=[c x-Ja,cy+a]
		 =c^2\lambda y+c^2v+cAa
		 =c\lambda\tilde y+c\pt{cv+\pt{A-\lambda \id }a},
\end{aligned}
\end{equation*}
where $\id \in\End(\abj)$ is the identity. 
This means that
\begin{equation}\label{changeTilde}
\tilde\lambda=c\lambda,\quad \tilde v=c\pt{cv+\pt{A-\lambda \id }a},\quad \tilde A=cA.
\end{equation}

\begin{remark}\label{unimod}
    The Lie algebra $\g$ is unimodular if and only if $\lambda+\tr(A)=0.$
    This of course does not depend on the choice of the presentation, as seen from \eqref{changeTilde}.
\end{remark}
For a fixed vector $y\in\ab\setminus\abj$ as above, and for every Hermitian metric $g_{\abj}$ on $\abj$, one can define a Hermitian metric $g$ on $\g$, as the unique extension of $g_{\abj}$ to $\g$ such that $\abs{y}_{g}=1=\abs{x}_{g}$ and $g(y,\abj)=0=g(x,\ab)$, where $x:=-Jy$.
We denote the set of all such metrics on $\g$ by $\mathcal{G}(y)$.

Conversely, given a Hermitian metric $g$ on $\g$, if $y$ is one of the two vectors $g$-orthogonal to $\abj$ in $\ab$ and of norm $1$, then $x:=-Jy$ is $g$-orthogonal to $\ab$ in $\g$ and has norm $1$.

\subsection{The Chevalley-Eilenberg differential}
We will now show how the action of the exterior derivative on forms can be written explicitly in terms of a presentation $(\lambda,v,A)$ defined by $y\in\ab\setminus \abj$.


The endomorphisms $A\in\End\pt{\abj}$ and $J\in\End(\g)$ induce endomorphisms on the dual spaces, $A^*\in\End\pt{\abj^*}$, $J\in\End(\g^*)$, defined by
\begin{equation}\label{a*}
\pt{A^*\Psi}(Y)=\Psi(AY),\quad\quad \pt{J\phi}(Z)=-\phi(JZ),
\end{equation}
for every $\Psi\in\abj^*,Y\in\abj,\phi\in\g^*,Z\in\g$.
We denote with the same symbol $A^*$ the extension of $A^*$ as derivation to $\Lambda^k\abj^*$ (in particular $A^*=0$ on $\R=\Lambda^0\abj^*$), whereas $J$ is defined on $\Lambda^k\g^*$ as the linear extension of $J(\phi_1\wedge\dots\wedge\phi_k)=(J\phi_1)\wedge\dots\wedge(J\phi_k)$, for all $\phi_1,\dots,\phi_k\in\g^*$.
In what follows, for $k\ge1$ and $Z\in\g$, $\iota_Z\colon\Lambda^k\g^*\to\Lambda^{k-1}\g^*$ will denote the contraction with $Z$. We will denote with the same symbols the $\C$-linear extension of endomorphisms of real vector spaces to their complexifications.

From the above definition it follows immediately that 
\begin{equation}\label{deriv}
  A^*(\iota_Y\Psi)=\iota_Y(A^*\Psi)-\iota_{AY}\Psi,
\end{equation}
for every $\Psi\in\Lambda^k\abj^*$ and $Y\in\abj$.
In other words,
\begin{equation}\label{iotabracket}
    \pq{\iota_Y,A^*}=\iota_{AY}.
\end{equation}

\begin{lemma}\label[lemma]{dk-form}
Let $(\g,\ab,J)$ be an almost abelian Lie algebra endowed with a complex structure, with a presentation $(\lambda,v,A)$ defined by $y\in\ab\setminus\abj$, and let $x:=-Jy$.
Then, for every $\vartheta\in\Lambda\g^*$, we have
\begin{equation*}
d\vartheta=-x^{\flat}\wedge y^\flat\wedge\pt{\iota_{\lambda y+v}\vartheta+A^*{\rest{\pt{\iota_{y}\vartheta}}{\abj}}}-x^\flat\wedge A^*\pt{\rest{\vartheta}{\abj}},
\end{equation*}
where $x^\flat,y^\flat$ are the $g$-duals of $x,y$ with respect to a metric $g\in\mathcal{G}(y)$.
\end{lemma}


\begin{proof}
Let $\vartheta\in\Lambda^k\g^*$ for some $k\ge0$. The formula is trivially satisfied for $k=0$, so we may assume $k\ge 1$.
Since $x^\flat$ spans a complement of $\ab^*$ in $\g^*$ and $x^\flat(x)=1$, we can write $d\vartheta=x^\flat\wedge(\iota_{x}d\vartheta)+\rest{d\vartheta}\ab$, with $\iota_{x}d\vartheta\in\Lambda^{k}\ab^*$.
Using the formula for the exterior derivative interpreted as the Chevalley–Eilenberg  differential operator,
\begin{equation}\label{ChevalleyEilenberg}
    \pt{d\vartheta}(X_1,\ldots,X_{k+1})=\sum_{1\le j<l\le k+1}(-1)^{j+l}\,\vartheta\pt{\pq{X_j,X_l},X_1,\ldots,\hat{X}_j,\hat{X}_l,\ldots,{X}_{k+1}},
\end{equation}
we see that $\rest{d\vartheta}\ab=0$, because $\ab$ is abelian.
Similarly,  $y^\flat$ spans a complement of $\abj^*$ in $\ab^*$, so  $\iota_{x}d\vartheta=y^\flat\wedge(\iota_{y}\iota_{x}d\vartheta)+\rest{\pt{\iota_{x}d\vartheta}}\abj$, with  $\iota_{y}\iota_{x}d\vartheta\in\Lambda^{k-1}\abj^*$.

Firstly, we show that $\rest{\pt{\iota_{x}d\vartheta}}\abj=-A^*\pt{\rest{\vartheta}{\abj}}$.
Since $A^*$ is defined on $\Lambda^k\abj^*$ extending the action on $\abj^*$ as a derivation, and the left-hand side is also a derivation because $\abj$ is abelian, it is enough to show this formula for $\vartheta\in\g^*$. 
If this is the case, for $Y\in\abj$,
\[
(\iota_{x}d\vartheta)(Y)=d\vartheta(x,Y)=-\vartheta(AY)=-\pt{A^*\pt{\rest{\vartheta}{\abj}}}(Y),
\]
as claimed.

Now, it remains to compute $\iota_{y}\iota_{x}d\vartheta$.
By \eqref{ChevalleyEilenberg}, 
\[
\iota_{y}\iota_{x}d\vartheta=-\iota_{[x,y]}\vartheta+\iota_{x}d(\iota_{y}\vartheta)=-\iota_{\lambda y+v}\vartheta-A^*\pt{\rest{\pt{\iota_{y}\vartheta}}{\abj}},
\]
where the last equality follows by the first part of the proof.
Combining the identities obtained above, we get
\begin{equation*}
d\vartheta=x^\flat\wedge(\iota_{x}d\vartheta)=-x^\flat\wedge y^\flat\wedge\pt{\iota_{\lambda y+v}\vartheta+A^*{\rest{\pt{\iota_{y}\vartheta}}{\abj}}}-x^\flat\wedge A^*\pt{\rest{\vartheta}{\abj}},
\end{equation*}
thus proving the statement.
\end{proof}

\begin{corollary}\label[corollary]{dkformJ}
If $(\g,\ab)$ is an almost abelian Lie algebra with complex structure $J$ and presentation $(\lambda,v,A)$ determined by $y\in\ab\setminus\abj$, then for every metric $g\in\mathcal{G}(y)$ and $\xi\in\Lambda\abj^*$, one has
\begin{equation}\label{dxi}
dx^\flat=0,\qquad dy^\flat=-\lambda x^\flat\wedge y^\flat,\qquad
d\xi=-x^\flat\wedge y^\flat\wedge{\iota_{v}\xi}-x^\flat\wedge A^*{\xi},
\end{equation}
where $x^\flat,y^\flat$ are the $g$-duals of $x:=-Jy$ and $y$.
\end{corollary}

If we consider the complexified Lie algebra $\g_\C:=\g\otimes \C$, we obtain similar formulas on the spaces of $(p,q)$-forms $\Lambda^{p,q}\g:=\Lambda^{p,q}\g_\C^*$. 
We can then rewrite the formulas obtained in \Cref{dkformJ} for complex forms.

\begin{corollary}\label{c13}
Let $(\g,\ab)$ be an almost abelian Lie algebra with complex structure $J$ and presentation $(\lambda,v,A)$ determined by $y=Jx\in\ab\setminus\abj$.
Let $x^\flat,y^\flat$ be the $g$-duals of $x,y$ for some $g\in\mathcal{G}(y)$, and
$\alpha:=\frac12(x^\flat+i\,y^\flat)\in\Lambda^{1,0}\g$.
Then, 
\begin{equation}\label{dd}
d\alpha=\lambda\alpha\wedge\bar\alpha,\quad\quad
d\tau=-2i\alpha\wedge\bar\alpha\wedge\iota_{v}\tau-\pt{\alpha+\bar\alpha}\wedge A^*\tau,
\end{equation}
 for every $\tau\in\Lambda_\C\abj^*:=\Lambda(\abj^*\otimes\C)$.
 In particular, $d(\alpha\wedge\bar\alpha)=0.$
\end{corollary}

\begin{proof}
    The statement follows immediately from the relation $x^\flat\wedge y^\flat=2i\alpha\wedge\bar\alpha.$
\end{proof}

We will use the  notations 
\[
d_v\tau=-2i\alpha\wedge\bar\alpha\wedge\iota_{v}\tau,\quad\quad
d_A\tau=-\pt{\alpha+\bar\alpha}\wedge A^*\tau,
\]
for $\tau\in\Lambda_\C\abj^*$, so we can split $\rest d\abj\colon\Lambda^{p,q}\abj\to \Lambda^{p+q+1}\g^*_\C$ as  $\rest d\abj=d_v+d_A$.

%

\medskip

By Corollary \ref{c13}, we can deduce formulas for the $J$-twisted differential operator $d^c$, defined as $d^c:=i\pt{\bar\partial-\partial}$, or, in our notations, $d^c=JdJ^{-1}$.

\begin{corollary}\label{ddcCor}
    Let $(\g,\ab)$ be an almost abelian Lie algebra with complex structure $J$ and presentation $(\lambda,v,A)$ determined by $y=Jx\in\ab\setminus\abj$.
Let $x^\flat,y^\flat$ be the $g$-duals of $x,y$, for some $g\in\mathcal{G}(y)$, and
$\alpha:=\frac12(x^\flat+i\,y^\flat)\in\Lambda^{1,0}\g$.
Then, 
\begin{equation}\label{dc}
d^c\alpha=i\lambda\alpha\wedge\bar\alpha,\quad\quad
d^c\tau=-2i\alpha\wedge\bar\alpha\wedge\iota_{Jv}\tau+i\pt{\alpha-\bar\alpha}\wedge A^*\tau,
\end{equation}
 for every $\tau\in\Lambda_\C\abj^*$.
 In particular, 
\begin{equation}\label{ddc}
dd^c\pt{\alpha\wedge\beta}=0,\quad\quad
dd^c\tau=-2i\alpha\wedge\bar\alpha\wedge\pt{\lambda A^*+A^* A^*}\tau,
\end{equation}
for all $\beta\in\Lambda\g^*_\C$, and $\tau\in\Lambda_\C\abj^*$.
\end{corollary}
\begin{proof}
    Using that $J\alpha=-i\alpha$, we obtain from \eqref{dd}:
    \[
d^c\alpha=JdJ^{-1}\alpha=iJd\alpha=iJ(\lambda\alpha\wedge\bar\alpha)=i\lambda\alpha\wedge\bar\alpha.
    \]
Since $J$ commutes with $A^*$ and satisfies $J(\iota_X\tau)=\iota_{JX} J\tau$, for every $X\in\abj^\C:=\abj\otimes \C$ and $\tau\in \Lambda_\C\abj^*$,
    \begin{eqnarray*}
d^c\tau&=&J(-2i\alpha\wedge\bar\alpha\wedge\iota_{v}(J^{-1}\tau)-\pt{\alpha+\bar\alpha}\wedge A^*(J^{-1}\tau))\\
&=&-2i\alpha\wedge\bar\alpha\wedge\iota_{Jv}\tau+i\pt{\alpha-\bar\alpha}\wedge A^*\tau,
    \end{eqnarray*}
    thus proving \eqref{dc}. 
    The first part of \eqref{ddc} follows then immediately from \eqref{dd} together with \eqref{dc}. 
    For the second part of \eqref{ddc}, using the second part of \eqref{dc} and \eqref{dd} we compute:
    \begin{eqnarray*}
    dd^c\tau&=&d(-2i\alpha\wedge\bar\alpha\wedge\iota_{Jv}\tau+i\pt{\alpha-\bar\alpha}\wedge A^*\tau)\\
    &=&id(\alpha-\bar\alpha)\wedge A^*\tau-i\pt{\alpha-\bar\alpha}\wedge d(A^*\tau)\\
    &=&2i(\alpha\wedge\bar\alpha)\wedge A^*\tau+i\pt{\alpha-\bar\alpha}\wedge\pt{\alpha+\bar\alpha}\wedge A^*A^*\tau\\
    &=&2i\alpha\wedge\bar\alpha\wedge\pt{\lambda A^*+A^* A^*}\tau.
    \end{eqnarray*}
\end{proof}

In what follows, it will be useful to have an explicit expression for the action of $A^*$ and $A^*A^*$ on Hermitian 2-forms.
For this, we will use the identification between endomorphisms and bilinear forms on a Euclidean vector space $(V,h)$, as follows.
For every endomorphism $M$ of $V$, the associated $h$-dual bilinear form is $M^\flat(\cdot,\cdot):=h(M\cdot,\cdot)$.
Then, we have:

\begin{lemma}\label{dualEnd}
Let $V$ be a real vector space  with a Hermitian structure $(J,h,\omega)$, and let $A$ be an endomorphism of $V$ commuting with $J$. 
Then,
\begin{equation}\label{ao}
    A^*\omega=\pt{2A_0J}^\flat,
\end{equation}
and
\begin{equation}\label{aao}
    A^*A^*\omega=\pt{2\pt{2A_0^2+\pq{A_0,A_1}}J}^\flat.
\end{equation}
where $A_0$ denotes the $h$-symmetric part of $A$ and $A_1$ the $h$-skew-symmetric part of $A$.
\end{lemma}

\begin{proof}
Since $A$ commutes with $J$, which is $h$-skew-symmetric, $A_0$ and $A_1$ also commute with $J$. We thus have, for every $X,Y\in V$:
\begin{eqnarray*}
(A^*\omega)(X,Y)&=&h(JAX,Y)+h(JX,AY)=h(JAX,Y)+h((A_0-A_1)JX,Y)\\&=&2h(A_0JX,Y),
\end{eqnarray*}
thus proving \eqref{ao}.
Using this,
\begin{eqnarray*}
(A^*A^*\omega)(X,Y)&=&A^*\omega(AX,Y)+A^*\omega(X,AY)=h(2A_0JAX,Y)+h(2A_0JX,AY)\\&=&h((2A_0^2J+2A_0A_1J)X,Y)+h(2(A_0-A_1)A_0JX,Y)\\
&=&h((4A_0^2J+2[A_0,A_1]J)X,Y),
\end{eqnarray*}
which proves \eqref{aao}.
\end{proof}

\section{\texorpdfstring{$p$}{p}-Kähler structures}\label{secpK}

Let $(V,J)$ be a $2n$-dimensional real vector space endowed with a complex structure.
For every basis $\gamma_1,\dots,\gamma_n$ of $\Lambda^{1,0}V^*$, we define the real $(n,n)$-form
\begin{equation}\label{vol}
    \vol=(i\gamma_1\wedge\bar\gamma_1)\wedge\dots\wedge (i\gamma_{n}\wedge\bar\gamma_{n}).
\end{equation}
Any real positive multiple of $\vol$ in $\Lambda^{n,n}V^*$ will be called a \textit{positive volume form} on $V$.
Note that if a different basis of $\Lambda^{1,0}V^*$ is chosen, then the $(n,n)$-form obtained as in \eqref{vol} is still a positive volume form on $V$, so the definition does not depend on the choice of the basis.

\begin{definition}
    Let $(V,J)$ be a $2n$-dimensional real vector space endowed with a complex structure.
    A (real) $(p,p)$-form $\Omega$ is called \textit{strictly weakly positive}, or \textit{transverse} if for every $J$-invariant subspace $W$ of $V$ of real dimension $2p$, the restriction $\Omega|_W$ is a positive volume form.
\end{definition}

An equivalent formulation is that, for all $\gamma_1,\dots,\gamma_{n-p}\in\Lambda^{1,0}V^*$ with $\gamma_1,\dots,\gamma_{n-p}\neq0$, 
\begin{equation*}
    \Omega\wedge (i\gamma_1\wedge\bar\gamma_1)\wedge\dots\wedge (i\gamma_{n-p}\wedge\bar\gamma_{n-p})
\end{equation*}
is a positive volume form on $V$.

\begin{definition}
    Let $J$ be a complex structure on a $2n$-dimensional real Lie algebra $\g$.
    For $1\le p\le n$, a \textit{$p$-Kähler structure} on $(\g,J)$ is a closed and transverse real $(p,p)$-form.
\end{definition}

Note that this definition is new only for $2\leq p\leq n-2$. 
Indeed, a $p$-Kähler structure is just a Kähler metric for $p=1$, a real volume form for $p=n$, whereas for $p=n-1$ we recover $(n-1)$-th powers of balanced metrics by \cite[Equation (4.8)]{michelsohn}. 
When $2\leq p\leq n-2$, we will refer to {\em genuine} $p$-Kähler structures.

In the next subsections we will consider separately the cases of $p$-Kähler structures for $p=1$ (Kähler), for $p=n-1$ (balanced), and for $2\leq p\leq n-2$ (genuine).

\subsection{Kähler metrics}

We recall the following result (first obtained in \cite{LR,fp21}): 
\begin{lemma}\label{l14}
    On an almost abelian Lie algebra $(\g,\ab)$ with complex structure $J$ and presentation $(\lambda,v,A)$ determined by $y\in\ab\setminus\abj$, a Hermitian metric $g\in\mathcal{G}(y)$ is Kähler if and only if $v=0$ and $A$ is skew-symmetric with respect to $g|_{\abj}$.
\end{lemma}
\begin{proof}
Let $\alpha:=\frac12(x^\flat+i\,y^\flat)\in\Lambda^{1,0}\g$. 
    Since $g(y,\abj)=0$, the fundamental form of $g$, defined by $\omega:=g(J\cdot,\cdot)$ can be  written as
    \[\omega=x^\flat\wedge y^\flat+\omega_0=2i\alpha\wedge\bar\alpha+\omega_0\] 
    where  $\omega_0\in \Lambda^{1,1}\abj$ is  the fundamental form of $g|_{\abj}$.
  By Corollary \ref{c13}, we thus have
    \[
d\omega=d\omega_0=-2i\alpha\wedge\bar\alpha\wedge\iota_{v}\omega_0-\pt{\alpha+\bar\alpha}\wedge A^*\omega_0.
    \]
    Comparing types, this shows that $\omega$ is closed if and only if $\iota_{v}\omega_0=0$ and $A^*\omega_0=0$. 
    The first condition is clearly equivalent to $v=0$, whereas by Lemma \ref{dualEnd}, $A^*\omega_0=0$ if and only if $A$ is skew-symmetric with respect to $g|_{\abj}$.
\end{proof}

We will now give a general criterion for an almost abelian Lie algebra $(\g,\ab)$ with complex structure $J$ to admit Kähler metrics in terms of a presentation $(\lambda,v,A)$. 

\begin{proposition}\label{lemma17}
    Let $J$ be a complex structure on an almost abelian Lie algebra $(\g,\ab)$ with presentation $(\lambda,v,A)$. 
    If $\g$ is non-unimodular, then there exists a Kähler metric on $(\g,J)$ if and only if $A$ is diagonalizable over $\C$ and has purely imaginary spectrum. If $\g$ is unimodular, then there exists a Kähler metric on $(\g,J)$ if and only if $A$ is diagonalizable over $\C$, has purely imaginary spectrum, and $v\in A(\abj)$.
\end{proposition}

\begin{proof}
Let $(\lambda,v,A)$ be determined by $y\in\ab\setminus\abj$.
Assume first that $g$ is a Kähler metric on $(\g,J)$ and let $\tilde y\in\ab\setminus\abj$ be the vector, unique up to sign, such that $g\in\mathcal{G}(\tilde y)$. Write $\tilde y=c y+a$ for some non-zero $c\in\R^*$ and $a\in\abj$, and let $(\tilde\lambda,\tilde v,\tilde A)$ be defined by \eqref{changeTilde}. 
By Lemma \ref{l14}, since $g$ is Kähler, we have $\tilde v=0$ and $\tilde A$ is skew-symmetric with respect to $g|_{\abj}$. By \eqref{changeTilde}, this means that $A$ is skew-symmetric as well with respect to $g|_{\abj}$ and $cv+(A-\lambda \id )a=0$. 
In particular, $A$ is diagonalizable over $\C$ and has purely imaginary spectrum.
Furthermore, if $\g$ is unimodular, then $\lambda=0$ by Remark \ref{unimod}, whence $v\in A(\abj)$ by the previous equation.

Conversely, assume that $A$ is diagonalizable over $\C$, has purely imaginary spectrum, and that moreover $v\in A(\abj)$, in the case where $\g$ is unimodular. Since $\tr(A)$ is real and imaginary in the same time, it has to vanish, so by Remark \ref{unimod}, $\g$ is unimodular if and only if $\lambda=0$. 
The assumption on $A$ ensures the existence of a Hermitian metric $g_0$ on $\abj$ such that $A$ is skew-symmetric with respect to $g_0$.
We also claim that there exists $a\in \abj$ such that $v+(A-\lambda \id )a=0$.
Indeed, when $\lambda\neq0$ this simply follows from the fact that $A-\lambda \id $ is invertible, and when $\lambda=0$, this follows from the assumption $v\in A(\abj)$, see also \cite[Remark 4.4]{fm}. 
Then, the vector $\tilde y\in\ab\setminus\abj$ defined by $\tilde y:=y+a$ determines a unique Hermitian metric $g\in\mathcal{G}(\tilde y)$ on $(\g,J)$ with $\rest{g}{\abj}=g_0$, and $g$ is Kähler by Lemma \ref{l14}.
\end{proof}

\begin{corollary}\label[corollary]{vinKer}
Let $J$ be a complex structure on a unimodular almost abelian Lie algebra $(\g,\ab)$ such that $(\g,J)$ does not admit $J$-compatible Kähler metrics.
Assume that $(\g,\ab,J)$ admits a presentation $(\lambda,v,A)$ such that $A$ is $g_0$-skew-symmetric for some metric $g_0$ on $\abj$.
Then, there exists a presentation $(0,\tilde v,A)$ with $0\neq \tilde v\in\ker(A)$.
In particular, $\det(A)=0$ for every presentation.
\end{corollary}

\begin{proof}
Let $(\lambda,v,A)$ be determined by $y\in\ab\setminus\abj$. Because $\g$ is unimodular and $A$ is $g_0$-skew-symmetric, we have $\lambda=0$ by Remark \ref{unimod}. 
Since the image of $A$ is the same as the image of $A^2$, there exists $a\in\abj$ such that $Av=-A^2a$. By \eqref{changeTilde}, the presentation $(\tilde \lambda,\tilde v,\tilde A)$ determined by $\tilde y:=y+a$ satisfies 
$\tilde\lambda=0$, $\tilde v={v+{A}a}$, $\tilde A=A$. Finally, Proposition \ref{lemma17} shows that the vector $v$ is not in $A(\abj)$, so $\tilde v\neq 0$, but $\tilde A\tilde v=0$ by the choice of $a$. 
\end{proof}

\subsection{Balanced metrics}

Like in the case of Kähler metrics, we will first recall the known criterion for a Hermitian metric being balanced in terms of the presentation that it determines.

\begin{lemma}[{\cite[Theorem 3.1]{fp23}}]\label{l44}
    On an almost abelian Lie algebra $(\g,\ab)$ with complex structure $J$ and presentation $(\lambda,v,A)$ determined by $y\in\ab\setminus\abj$, a Hermitian metric $g\in\mathcal{G}(y)$ is balanced if and only if $v=0$ and $\tr A=0$.
    In particular, if this is the case, every metric in $\mathcal{G}(y)$ is balanced.
\end{lemma}

\begin{proof}
    As above, we write the fundamental form of $g\in\mathcal{G}(y)$, $\omega:=g(J\cdot,\cdot)$ as
    \[\omega=x^\flat\wedge y^\flat+\omega_0=2i\alpha\wedge\bar\alpha+\omega_0\] 
    with $\omega_0\in \Lambda^{1,1}\abj$. By Corollary \ref{c13} we thus have
    \begin{equation}\label{domegan-1}
d(\omega^{n-1})=d(\omega_0^{n-1})=-2i\alpha\wedge\bar\alpha\wedge\iota_{v}(\omega_0^{n-1})-\pt{\alpha+\bar\alpha}\wedge A^*(\omega_0^{n-1}).
    \end{equation}
    Comparing types, this shows that $\omega^{n-1}$ is closed if and only if $\iota_{v}(\omega_0^{n-1})=0$ and $A^*(\omega_0^{n-1})=0$. Since $\omega_0^{n-1}$ is a volume form of $\abj$, the first condition is equivalent to $v=0$, whereas the second condition is equivalent to $\tr A=0$.
\end{proof}

We now give a criterion for the existence of balanced metrics which only depends on some arbitrary presentation.

\begin{proposition}\label{37}
    Let $J$ be a complex structure on an almost abelian Lie algebra $(\g,\ab)$ with presentation $(\lambda,v,A)$. Then, there exists a balanced metric on $(\g,J)$ if and only if $\tr A=0$ and $v$ belongs to the image of $A-\lambda \id $. 
\end{proposition}

\begin{proof}
Let $(\lambda,v,A)$ be determined by $y\in\ab\setminus\abj$.
Assume first that $g$ is a balanced metric on $(\g,J)$ and let $\tilde y\in\ab\setminus\abj$ be the vector, unique up to sign, such that $g\in\mathcal{G}(\tilde y)$. 
Write $\tilde y=c y+a$ for some $c\in\R^*$ and $a\in\abj$, and let $(\tilde\lambda,\tilde v,\tilde A)$ be defined by \eqref{changeTilde}. 
By Lemma \ref{l44}, since $g$ is balanced we have $\tilde v=0$ and $\tr\tilde A=0$. 
By \eqref{changeTilde}, this means that $\tr A=0$ and $cv+(A-\lambda \id )a=0$, whence $v\in (A-\lambda \id )(\abj)$.

Conversely, assume that $\tr A=0$, and that $v\in (A-\lambda \id )(\abj)$, so there exists $a\in \abj$ such that $v+(A-\lambda \id )a=0$.
Then, the presentation $(\tilde\lambda,\tilde v,\tilde A)$ determined by the vector $\tilde y:=y+a\in\ab\setminus\abj$ satisfies $\tilde v=0$ and $\tr \tilde A=0$ by \eqref{changeTilde}. Consequently, every Hermitian metric $g\in\mathcal{G}(\tilde y)$ on $(\g,J)$ is balanced by Lemma \ref{l44}.
\end{proof}

\subsection{Genuine \texorpdfstring{$p$}{p}-Kähler structures}

Let $(\g,\ab)$ be an almost abelian Lie algebra of real dimension $2n$ with complex structure $J$, and $\Omega$ a real $(p,p)$-form on $(\g,J)$. For every presentation $(\lambda,v,A)$ defined by $y=Jx\in\ab\setminus \abj$, and for every $g\in\mathcal{G}(y)$, we define the $(1,0)$-vector $X:=x-iy$, and the $(1,0)$-form $\alpha:=\frac12(x^\flat+iy^\flat)\in\Lambda^{1,0}\g$ as above.
We can write
\begin{equation}\label{dOmega}
\Omega=\Omega_J+\alpha\wedge\eta+\bar\alpha\wedge\bar\eta+i\,\alpha\wedge\bar\alpha\wedge\Phi,    
\end{equation}
with $\Omega_J:=\rest\Omega\abj\in\Lambda^{p,p}\abj$, $\eta:=\rest{\iota_{X}\Omega}\abj\in\Lambda^{p-1,p}\abj$, and $\Phi:=i\,\iota_{ X}\iota_{\bar X}\Omega\in\Lambda^{p-1,p-1}\abj$.
By \cite[Theorem 3.4]{FagMai}, if $\Omega$ is transverse, so are $\Omega_J$ and $\Phi$.
Using Corollary \ref{c13}, we then compute
\begin{eqnarray*}
    d\Omega&=&d\Omega_J+{\lambda}\alpha\wedge\bar\alpha\wedge\pt{\eta-\bar\eta}-\alpha\wedge \da \eta-\bar\alpha\wedge\da\bar\eta\\
            &=&-\pt{\alpha+\bar\alpha}\wedge A^*\Omega_J-i\alpha\wedge\bar\alpha\wedge\pt{
                2\iota_v\Omega_J+i\pt{\lambda \id +A^*}\pt{\eta-\bar\eta}
                },
\end{eqnarray*}
where $\id $ is the identity on $\Lambda_\C\abj$.
This shows that $\Omega$ is closed if and only if 
\begin{equation}\label{domega=0}
        A^*\Omega_J=0,\quad\quad 2\iota_v\Omega_J=-i\pt{\lambda \id +A^*}\pt{\eta-\bar\eta}.
    \end{equation}
We can now prove a preliminary result, stating that an almost abelian algebra with complex structure admitting a genuine $p$-Kähler structure also admits a Kähler metric, provided that the endomorphism $A$ of $\abj$ appearing in some presentation is diagonalizable over $\C$.

\begin{proposition}\label{AdiagK}
    Let $(\g,\ab)$ be an almost abelian Lie algebra of real dimension $2n$ endowed with a complex structure $J$, with a presentation $(\lambda,v,A)$.
    Assume that $A$ is diagonalizable over $\C$. Then, if $(\g,J)$ admits a $p$-Kähler structure for some integer $p$ satisfying $2\leq p\leq n-2$, it also admits a Kähler metric. 
\end{proposition}

\begin{proof}
    Let $\Omega$ be a $p$-Kähler structure on $(\g,J)$, and define $\Omega_J,\eta,\Phi$ as in  \eqref{dOmega}, with $A^*\Omega_J=0$ by \eqref{domega=0}. 
    Since $A$ commutes with $J$, $A$ preserves the spaces  $\pt{\abj^\C}^{1,0}$ and $\pt{\abj^\C}^{0,1}$.
    Let $X_2,\dots,X_{n}$ be a basis of $\pt{\abj^\C}^{1,0}$ of eigenvectors of $A$, with eigenvalues $z_2,\dots,z_{n}$ respectively.

    Then, for every $2\le j_1<\dots<j_p\le n$,
    \begin{equation*}
       0= A^*\Omega_J\pt{X_{j_1},\bar X_{j_1},\dots,X_{j_p},\bar X_{j_p}}=\Omega\pt{X_{j_1},\bar X_{j_1},\dots,X_{j_p},\bar X_{j_p}}\sum_{k=1}^p \pt{z_{j_k}+\bar z_{j_k}} .
    \end{equation*}
    Since $\Omega$ is transverse, $\Omega\pt{X_{j_1},\bar X_{j_1},\dots,X_{j_p},\bar X_{j_p}}$ cannot be $0$, so we have
    \begin{equation*}
        \sum_{k=1}^p \pt{z_{j_k}+\bar z_{j_k}}=0.
    \end{equation*}
    Since this holds for every choice of $j_1,\dots,j_p$, and $p<n-1$, it follows that $z_j+\bar z_j=0$, for every $j=2,\dots,n$.
    In other words, the eigenvalues of $A$ are all purely imaginary, so we can choose a Hermitian metric $g_0$ on $\abj$ such that $A$ is $g_0$-skew-symmetric. 
    If $\g$ is not unimodular, the thesis follows directly from Proposition \ref{lemma17}.
    
    If $\g$ is unimodular, Remark \ref{unimod} implies that $\lambda=0$.
    For the sake of contradiction, assume $(\g,J)$ does not admit compatible Kähler metrics.
    By \Cref{vinKer}, we can assume $0\neq v\in\ker A$, and thus $AJv=0$ as well, since $A$ commutes with $J$.
    We will now use the second equation in \eqref{domega=0}, stating that $\iota_v\Omega_J=A^*\tau$, for some $\tau\in\Lambda^{2p-1}\abj^*$, so that
    \begin{equation*}
       \Omega\pt{v,Jv,X_{2},\bar X_{2},\dots,X_{p},\bar X_{p}}=\pt{A^*\tau}\pt{Jv,X_{2},\bar X_{2},\dots,X_{p},\bar X_{p}}.
    \end{equation*}
    Now, the left hand side is non-zero due to the positivity of $\Omega_J$, whereas the right hand side vanishes because $AJv=0$ and the eigenvalues of $A$ are imaginary, giving a contradiction.
\end{proof}

We consider now the general case, where $A$ is not necessarily diagonalizable over $\C$.
In this case, we will identify the $\C$-vector space $(\abj,J)$ with $\pt{\abj^\C}^{1,0}$, and $A$ with the restriction of its $\C$-linear extension to $\pt{\abj^\C}^{1,0}$ (recall that $A$ commutes with $J$). 

Let $\pg{X_2,\dots,X_n}\subset\pt{\abj^\C}^{1,0}$ be a 
basis of $\pt{\abj^\C}^{1,0}$ so that the matrix of $A\in\End\pt{\abj^\C}^{1,0}$ with respect to this basis is in its Jordan normal form. This means that there exist complex numbers $z_2,\dots,z_n\in\C$, and $\delta_2,\dots,\delta_{n-1}\in\pg{0,1}$ such that 
 $A X_n=z_nX_n$ and
    \begin{equation}\label{jordandual}
    AX_k=z_kX_k+\delta_{k}X_{k+1},\qquad 
    \delta_{k}\pt{z_{k}-z_{k+1}}=0, \qquad \forall \ k=2,\dots,n-1.
\end{equation}
Let also $\alpha_2,\dots,\alpha_n\in\Lambda^{1,0}\abj$ be the dual of the Jordan basis $\pg{X_2,\dots,X_n}$. By duality we readily obtain
$A^*\alpha_2=z_2\alpha_2$ and
\begin{equation}\label{jordan}
    A^*\alpha_k=z_k\alpha_k+\delta_{k-1}\alpha_{k-1}, \qquad 
    \delta_{k-1}\pt{z_{k-1}-z_{k}}=0, \qquad \forall\ k=3,\dots,n.
\end{equation}
We will call $\pg{\alpha_2,\dots,\alpha_n}$ a dual Jordan basis for $A^*$.

For $k=2,\dots,n$, denote $\iota_k:=\iota_{X_k}$ and $\bar\iota_k:=\iota_{\bar X_k}$.
Then, \eqref{iotabracket} reads
\begin{equation}\label{iotajA}
    \pq{\iota_k,A^*}=z_k \iota_k+\delta_k\iota_{k+1},
    \qquad 
    \pq{\bar\iota_k,A^*}=\bar z_k  \bar\iota_k+\delta_k\bar\iota_{k+1},
\end{equation}
where here and below, we denote $\delta_{k}=0$ and $\iota_{k+1}=\bar\iota_{k+1}=0$, for every $k\geq n$.

In the following, starting from a genuine $p$-Kähler structure $\Omega$, we will use suitable restriction and contraction operators to construct transverse forms of lower degree, or of the same degree, but defined on subspaces of smaller dimension, ultimately obtaining fundamental forms of some Hermitian metrics.
The closedness will imply that such forms are preserved by restrictions of $A^*$, and a key point in our argument will be the following result.

\begin{lemma}\label{l22}
Let $V$ be a $2k$-dimensional real vector space  with a Hermitian structure $(J,h,\omega)$, and let $M$ be any endomorphism of $V$ commuting with $J$. 
If $M^*\omega=t\omega$, for some $t\in\R$, then $M$ is diagonalizable over $\C$.
\end{lemma}

\begin{proof}
Let $M_0,M_1$ be the $h$-symmetric and $h$-skew-symmetric components of $M$, respectively. By Lemma \ref{dualEnd} we get $M_0=\frac t2\id$. 
In particular, $M$ is diagonalizable over $\C$, and all its eigenvalues have real part $\frac t2$.
\end{proof}

In the proof of the first main result, an important role will be played by a dimensionality reduction argument, where we will use a recursive formula for the action of $A^*$ on iterated contractions of a fixed form.

\begin{lemma}\label{induction}
    Let $V$ be a real vector space of dimension $2(n-1)$ endowed with a complex structure $J$.
    Let $A$ be an endomorphism of $V$ commuting with $J$, and let $\pg{X_2,\dots,X_n}\in(V^\C)^{1,0}$ be a Jordan basis for $A$, satisfying \eqref{jordandual}.
    For a real form $\psi_0\in\Lambda^{q,q}V^*$, we denote $\phi_0:=A^*\psi_0$, and we define recursively $\psi_j,\phi_j\in\Lambda^{q-j,q-j}V^*$ by
\begin{equation}\label{psijphij}
    \psi_{j+1}:=i\,\iota_{n-j}\bar\iota_{n-j}\psi_{j},\quad
    \phi_{j+1}:=i\,\iota_{n-j}\bar\iota_{n-j}\phi_{j},\qquad\forall0\le j\le q-1,
\end{equation}
Where we denote as before by $\iota_k:=\iota_{X_k}$ and $\bar\iota_k:=\iota_{\bar X_k}$ for $k=2,\dots,n$.
Let $u_j:=z_j+\bar z_j$, and define recursively $U_0=0$, $U_{j+1}=U_j-u_{n-j}$.
Then, 
\begin{equation}\label{Apsij}
    A^*\psi_j=U_j\psi_j+\phi_j,\qquad\forall0\le j\le q.
\end{equation}
\end{lemma}

\begin{proof}
Equation \eqref{Apsij} clearly holds for $j=0$. 
We now assume that \eqref{Apsij} holds for some $j\geq 1$.
By \eqref{jordandual}, we get
\begin{equation*}
    \iota_{AX_{n-j}}\psi_j=\iota_{(z_{n-j}X_{n-j}+\delta_{n-j}X_{n-j+1})}\psi_j=z_{n-j}\iota_{n-j}\psi_j,
\end{equation*}
as $\iota_{n-j+1}\psi_j=0$ by \eqref{psijphij} (applied to $j-1$). By complex conjugation we also get $\iota_{A\bar X_{n-j}}\psi_j=\bar z_{n-j}\bar\iota_{n-j}\psi_j.$
Then, by \eqref{iotajA} we can write
\begin{eqnarray*}
    A^*\psi_{j+1}&=&A^*\pt{i\,\iota_{n-j}\bar\iota_{n-j}\psi_{j}}\\
    &=&i\,\iota_{n-j}\pt{A^*\pt{\bar\iota_{n-j}\psi_{j}}  }-i\, z_{n-j}\iota_{n-j}\bar\iota_{n-j}\psi_{j}\\
    &=&i\,\iota_{n-j}\bar\iota_{n-j}\pt{A^*\psi_{j}}-i\bar{z}_{n-j}\,\iota_{n-j}\bar\iota_{n-j}\psi_{j}-i\,z_{n-j} \iota_{n-j}\bar\iota_{n-j}\psi_{j}\\
    &=&i\,\iota_{n-j}\bar\iota_{n-j}(U_j\psi_j+\phi_j)-i\,u_{n-j}\iota_{n-j}\bar\iota_{n-j}\psi_j\\
    &=&U_{j+1}\psi_{j+1}+\phi_{j+1},
\end{eqnarray*}
thus proving the induction step.
\end{proof}

We are now ready to prove the following characterization of the existence of $p$-Kähler structures, generalizing \cite[Theorem 4.2]{fm}.

\begin{theorem}\label{pKahler}
    Let $J$ be a complex structure on an almost abelian Lie algebra $(\g,\ab)$ of real dimension $2n\ge8$.
    For every $p\le n-2$, if $(\g,J)$ admits a $p$-Kähler structure,  then it admits a Kähler metric.     
\end{theorem}

\begin{proof}
The proof will consist of two main steps.
In the first one, we will prove that whenever $(\g,\ab,J)$ admits a $p$-Kähler structure $\Omega$, for every presentation $(\lambda,v,A)$, $\abj^\C$ can be written as $\abj^\C=V+ W$, such that $A$ is diagonalizable on $V$ and preserves $W$, so that $A^*$ preserves $W^*$.
In the second step, we will prove that we can associate to $\Omega$ a Hermitian metric $\omega$ on $W$ such that $A^*\omega=0$.
This will allow us to conclude, using Lemma \ref{l22} and Proposition \ref{AdiagK}.

Let $\Omega$ be a $p$-Kähler structure on $(\g,J)$, and $\Omega_J$ defined as in \eqref{dOmega}.
For  a presentation $(\lambda,v,A)$  of $(\g,\ab,J)$, we have by  \eqref{domega=0} that $A^*\Omega_J=0.$
Fix a Jordan basis $\pg{X_2,\dots,X_n}\subset(\abj^\C)^{1,0}$ for $A$ satisfying \eqref{jordandual}, with dual basis $\pg{\alpha_2,\dots,\alpha_n}$.
For every $k=1,\dots,n-1$, we will consider the $2k$-dimensional vector subspaces of $\abj^\C$: 
\begin{equation*}
\begin{aligned}
    & V_{k}:=\operatorname{span}\pg{X_2,\dots, X_{k+1},\bar X_2,\dots, \bar X_{k+1}},\\
    & W_k:=\operatorname{span}\pg{X_{n-k+1},\dots, X_n,\bar X_{n-k+1},\dots, \bar X_n}.
\end{aligned}
\end{equation*}
Since $X_j$ is a $(1,0)$-vector for all $j=2,\dots,n$, the complex vector spaces $V_k$ and $W_k$ are  complexifications of real vector subspaces $V_k^\R$ and $W_k^\R$ of $\abj$ which are $J$-invariant and have real dimension $2k$.
Moreover, $\abj^\C=V_{k}\oplus W_{n-k-1}$, for all $1\le k\le n-1$.
Corresponding to this splitting, we will identify the duals of $V_k$ and $W_k$ with the following subspaces of $(\abj^\C)^*$:
\begin{equation*}
\begin{aligned}
    & V_{k}^*=\operatorname{span}\pg{\alpha_2,\dots, \alpha_{k+1},\bar \alpha_2,\dots, \bar \alpha_{k+1}},\\
    & W_k^*:=\operatorname{span}\pg{\alpha_{n-k+1},\dots, \alpha_n,\bar \alpha_{n-k+1},\dots, \bar \alpha_n},
\end{aligned}
\end{equation*}
and again $(\abj^\C)^*=V_{k}^*\oplus W_{n-k-1}^*$, for all $1\le k\le n-1$.
Note that by \eqref{jordandual} $A$ preserves each $W_k$ and by \eqref{jordan}, $A^*$ preserves each $V_k^*$.
Furthermore, $A$ preserves $V_k$ if and only if $\delta_{k+1}=0$, and similarly $A^*$ preserves $W_k^*$  if and only if $\delta_{n-k}=0$.

For $j=0,\dots,p-1$, we denote $\psi_0\coloneqq\Omega_J$, and define recursively the real forms $\psi_j,\phi_j\in\Lambda^{p-j,p-j}V_{n-j-1}^*$  as in Lemma \ref{induction}.
By construction, all the $\psi_j$'s are transverse, because $\Omega$ is transverse.
Then, by Lemma \ref{induction}, we get
\begin{equation}\label{Apsij0}
    A^*\psi_j=U_j\psi_j,
\end{equation}
for all  $j=0,\dots,p-1$, where $U_j=-\pt{u_n+\dots+u_{n-j+1}}$, and $u_k=z_k+\bar z_k\in\R$, for all $k=2,\dots,n$.
For $j=p-1$, $\psi_{p-1}$ is actually the fundamental form of a Hermitian metric $h$ on $V_{n-p}$, 
so we are almost in the assumptions of Lemma \ref{l22}, except that a priori $A$ does not preserve $V_{n-p}$. In order to overcome this difficulty, we introduce the endomorphism $M$ of $V_{n-p}$ defined by
\begin{equation}\label{endM}
    MX_l=AX_l=z_lX_l+\delta_{l}X_{l+1},\quad l=2,\dots,n-p,\quad M X_{n-p+1}=z_{n-p+1}X_{n-p+1},
\end{equation}
and commuting with the complex conjugation in $V_{n-p}$. Then, $M^*=\rest{A^*}{V_{n-p}^*}$, so by \eqref{Apsij0} for $j=p-1$, together with Lemma \ref{l22}, we conclude that $M$ is diagonalizable on $V_{n-p}$.
This shows that 
\begin{equation}\label{delta}
    \delta_l=0,\qquad\forall\ l=2,\dots,n-p.
    \end{equation}
Consequently, $A$ preserves $V_{n-p-1}$, its restriction to $V_{n-p-1}$ is diagonalizable, and $A^*$ preserves $W_k^*$, for all $k=p,\dots,n-1$.

For the second part of the proof, we consider the restriction of $\Omega$ to $W_{p+1}$ (note that this space is well defined since $p\le n-2$ by assumption):
\begin{equation*}
    \rest{\Omega}{W_{p+1}}=\rest{\Omega_J}{W_{p+1}}\in\Lambda^{p,p}W_{p+1}^*.
\end{equation*}
Thus, $\rest{\Omega_J}{W_{p+1}}$ is a transverse $(p,p)$-form on the $(p+1)$-dimensional complex vector space $(W_{p+1}^\R,J)$, and so by \cite[Equation (4.8)]{michelsohn}, there exists a Hermitian metric on $(W_{p+1}^\R,J)$ with fundamental form $\sigma$ such that $\rest{\Omega_J}{W_{p+1}}=\sigma^p$.

From the first part of the proof we know that $A^*$ preserves $W_{p+1}^*$, and $A^*|_{W_{p+1}^*}=A_{p+1}^*$, where $A_{p+1}:=\rest{A}{W_{p+1}}$ (recall that $A$ preserves each $W_{k}$), so $A^*\Omega_J=0$ implies
\begin{equation}\label{Asigmap}
    0=\rest{\pt{A^*\Omega_J}}{W_{p+1}}=A_{p+1}^*\pt{\rest{\Omega_J}{W_{p+1}}}=A_{p+1}^*\pt{\sigma^p}=p\pt{A_{p+1}^*\sigma}\wedge\sigma^{p-1},
\end{equation}
where we used in the first equality that $A^*\Omega_J=0$, by \eqref{domega=0}.
Since we are now in complex dimension $p+1$, the linear map 
\begin{equation}\label{Lsigmaiso}
   \cdot \wedge\sigma^{p-1}\colon\Lambda^2 W_{p+1}^*\to\Lambda^{2p}W_{p+1}^*
\end{equation}
is an isomorphism, so we conclude by \eqref{Asigmap} that $A_{p+1}^*\sigma=0.$
Lemma \ref{l22} can then be applied to $A_{p+1}$ which is now a well defined endomorphism of $W_{p+1}^\R$ commuting with $J$.
We conclude that $A$ is diagonalizable on $W_{p+1}$, and thus $\delta_j=0$, for $j=n-p,\dots,n-1$.
Together with \eqref{delta}, this means that $A$ is diagonalizable on $\abj^\C$.
The assumptions of Proposition \ref{AdiagK} are thus satisfied, so we conclude that $(\g,J)$ admits a Kähler metric.
\end{proof}

\section{\texorpdfstring{$p$}{p}-pluriclosed structures}\label{secpPl}

 A $p$-pluriclosed structure on a Lie algebra $\g$ with complex structure $J$ of real dimension $2n$ is a transverse $(p,p)$-form $\Omega$ such that $dd^c\Omega=0$. 
For $p=1$ this corresponds to pluriclosed metrics. For $p=n-1$, every transverse $(n-1,n-1)$-form $\Omega$ can be written as $\Omega=\omega^{n-1}$ for some Hermitian form $\omega$ by \cite[Equation (4.8)]{michelsohn}, so $\Omega$ is $(n-1)$-pluriclosed if and only if $\omega$ is Gauduchon. 

We will study separately in the next subsections the cases $p=1$ (pluriclosed), $p=n-1$ (Gauduchon) and $2\leq p\leq n-2$ (genuine $p$-pluriclosed).

\subsection{Pluriclosed metrics}
Similarly to the Kähler case, we can use Corollary \ref{ddcCor} to recover the following characterization of pluriclosed metrics, as in \cite{AL}.

\begin{lemma}\label{lskt}
    On an almost abelian Lie algebra $(\g,\ab)$ with complex structure $J$ and presentation $(\lambda,v,A)$ determined by $y\in\ab\setminus\abj$, a Hermitian metric $g\in\mathcal{G}(y)$ is pluriclosed if and only if
     \begin{equation}\label{p(A0)}
         \lambda A_0+2A_0^2+\pq{A_0,A_1}=0,
     \end{equation}
     where $A_0$ denotes the $g$-symmetric part of $A$ and $A_1$ the $g$-skew-symmetric part of $A$.
\end{lemma}

\begin{proof}
As in the proof of Lemma \ref{l14}, the fundamental form $\omega:=g(J\cdot,\cdot)$ can be  written 
    \[\omega=x^\flat\wedge y^\flat+\omega_0=2i\alpha\wedge\bar\alpha+\omega_0\] 
    with $\omega_0=\rest{\omega}{\abj}\in \Lambda^{1,1}\abj$. 
Then, by \eqref{ddc}, 
\begin{equation}\label{ddcomega}
    dd^c\omega=dd^c\omega_0=-2i\alpha\wedge\bar\alpha\wedge\pt{\lambda A^*+A^* A^*}\omega_0,
\end{equation}
so $\omega$ is pluriclosed if and only if $\pt{\lambda A^*+A^* A^*}\omega_0=0$.
Using Lemma \ref{dualEnd}, we obtain equation \eqref{p(A0)}.
\end{proof}

\begin{remark}\label{rmkSKT} If we denote by $A^t:=A_0-A_1$ the adjoint of $A$ with respect to $g$,
equation \eqref{p(A0)} is equivalent to the endomorphism
    \begin{equation}\label{p(A)}
        \lambda A+A^2+A^tA
    \end{equation}
being $g$-skew, so Lemma \ref{lskt} recovers the results of \cite[Lemma 4.2]{AL}.
\end{remark}

In order to exploit this relation  we will need the following results of linear algebra:

\begin{lemma}\label{polyn}
    Let $S,M$ be endomorphisms on a vector space $V$, and assume $S$ is symmetric with respect to some metric $g$.
If there is some polynomial $P$ such that
    \begin{equation}
        P(S)+[S,M]=0,
    \end{equation}
    then $[S,M]=0.$
\end{lemma}

\begin{proof}
    Since by assumption $P(S)=[M,S]$, 
    for every polynomial in one variable $Q$, we have $P(S)Q(S)=[M,S]Q(S)$, so
    \begin{equation*}
        \tr\pt{P(S)Q(S)}=\tr\pt{[M,S]Q(S)}=\tr\pt{MSQ(S)}-\tr\pt{SMQ(S)}=0.
    \end{equation*}
    In particular, $\tr\pt{P(S)^k}=0$, for every $k\ge1$.
    Thus, $P(S)$ is nilpotent.
    On the other hand, if $S$ is $g$-symmetric, then so is $P(S)$, whence $P(S)=0$, and the lemma is proved.
\end{proof}

\begin{proposition}\label{PAomega=0}
Let $(V,J)$ be a real vector space endowed with a complex structure and $A\in\End(V)$, commuting with $J$.
    Let $P$ be a monic polynomial of degree $2$ with real coefficients, and assume there exists a Hermitian metric on $(V,J)$ with fundamental form $\omega$ such that $P(A^*)\,\omega=0$.
    Then, $A$ is diagonalizable over $\C$, and for every complex eigenvalue $z\in\mathrm{Spec}_\C(A)$, we have $P(z+\bar z)=0.$
\end{proposition}

\begin{proof}
    We denote the real dimension of $V$ by $2(n-1)$ and consider a Jordan basis $\pg{X_2,\dots,X_{n}}$ for $A$, satisfying \eqref{jordandual}. For $k=2,\dots,n$, we denote as before $\iota_k:=\iota_{X_k}$ and $\bar\iota_k:=\iota_{\bar X_k}$ Writing $P(s)=s^2+as+b$, for $a,b\in\R$, we obtain for every $j,k=2,\dots,n$:
\begin{equation}
\begin{aligned}
    0=\iota_j\bar \iota_kP(A^*)\,\omega=&\iota_j\bar \iota_k\pt{A^*A^*+aA^*+b\id }\omega\\
    =&\pt{\pq{\iota_j\bar \iota_k,A^*A^*}+a\pq{\iota_j\bar \iota_k,A^*}+b\iota_j\bar \iota_k}\omega,
\end{aligned}
\end{equation}
as $A^*\iota_j\bar \iota_k\omega\in A^*\C=0$.
Then, combining \eqref{iotajkA} and \eqref{iotajkAA2} proved in the Appendix, we get for all $j,k=2,\dots,n$:
\begin{equation}\label{jkaa=0}
\begin{aligned}
0=&\pt{\pq{\iota_j\bar\iota_k,A^*A^*}+a\pq{\iota_j\bar\iota_k,A^*}+b\iota_j\bar \iota_k}\omega\\
    =& \pt{\pt{z_j+\bar z_k}^2+a\pt{z_j+\bar z_k}+b}\iota_j\bar\iota_k\omega\\ 
    &+\pt{2\pt{z_j+\bar z_k}+a}\pt{\delta_j\iota_{j+1}\bar \iota_k\omega+\delta_k\iota_j\bar \iota_{k+1}\omega}\\
    &+2\delta_j\delta_k \iota_{j+1}\bar \iota_{k+1}\omega+\delta_j\delta_{j+1}\iota_{j+2}\bar \iota_{k}\omega+\delta_k\delta_{k+1}\iota_{j}\bar \iota_{k+2} \omega.
\end{aligned}
\end{equation}
For $j=k=n$, we obtain
\begin{equation*}
    \pt{\pt{z_n+\bar z_n}^2+a\pt{z_n+\bar z_n}+b}\iota_n\bar\iota_n\omega=0,
\end{equation*} 
and since $\iota_n\bar\iota_n\omega\neq0$, this implies that
\begin{equation}\label{Pun=0}
    0=\pt{z_n+\bar z_n}^2+a\pt{z_n+\bar z_n}+b=P(z_n+\bar z_n).
\end{equation}
For every complex eigenvalue $z$ of $A$ we can assume, up to reordering $\pg{X_2,\dots,X_{n}}$, that $z_n=z$ or $z_n=\bar z$, so the last part of the statement is proved. 
We will now show that each Jordan block of $A$ has size $1$, proving that $A$ is diagonalizable over $\C$.

Indeed, assume for the sake of contradiction that $A$ has a Jordan block of rank larger than $1$. 
Then, up to reordering $\pg{X_2,\dots,X_{n}}$, we have $z_{n-1}=z_n$, $\delta_{n-1}=1$.
Equation \eqref{jkaa=0} for $j=n-1,k=n$ reads
\begin{equation*}
\pt{\pt{z_n+\bar z_n}^2+a\pt{z_n+\bar z_n}+b}\iota_{n-1}\bar\iota_n\omega+\pt{2(z_n+\bar z_n)+a}{\iota_{n}\bar \iota_{n}\omega}=0
    ,
\end{equation*}
as $\delta_k=0=\delta_{j+1}$.
Thus, by \eqref{Pun=0}, $2(z_n+\bar z_n)+a=0$.
We can now apply \eqref{jkaa=0} for $j=k=n-1$ to get
\begin{equation*}
0=2\delta_{n-1}\delta_{n-1} \iota_{n}\bar \iota_{n}\omega=2{\iota_{n}\bar \iota_{n}\omega},
\end{equation*}
as $\delta_{n-1}=\delta_{n-1}=1$ and $\delta_{n}=0$.
However, this contradicts the positivity of $\omega$, ultimately allowing us to conclude.
\end{proof}

As a consequence, we obtain the following characterization of pluriclosed metrics.

\begin{corollary}[{\cite[Theorem 4.6]{AL}}]\label{corSKT}
    On an almost abelian Lie algebra $(\g,\ab)$ with complex structure $J$ and presentation $(\lambda,v,A)$ determined by $y\in\ab\setminus\abj$, a Hermitian metric $g\in\mathcal{G}(y)$ is pluriclosed if and only if 
    \begin{equation}\label{eqSpec}
        [A^t,A]=0,\quad\text{and}\quad \real\pt{\operatorname{Spec}_\C(A)}\subseteq\pg{0,-\frac{\lambda}{2}},
    \end{equation}
    where $\cdot^t$ is transposition with respect to $g$.
\end{corollary}

\begin{proof}
Let $g\in\mathcal{G}(y)$ be pluriclosed, let $\omega$ be the associated fundamental form and $\omega_0=\rest{\omega}{\abj}$. 
Then, by \eqref{ddcomega}, we have $\pt{\lambda A^*+A^* A^*}\omega_0=0$, so that we are in the assumptions of Proposition \ref{PAomega=0}.
It follows that $A$ is diagonalizable over $\C$ and for every eigenvalue $z\in\C$, we have $\lambda(z+\bar z)+(z+\bar z)^2=0$, namely the second part of \eqref{eqSpec} holds.
The first part of \eqref{eqSpec} is a straightforward consequence of equation \eqref{p(A0)} and Lemma \ref{polyn}, as $[A^t,A]=2[A_0,A_1]$.

Conversely, fix  $g\in\mathcal{G}(y)$ and assume \eqref{eqSpec} holds, where $A^t$ is determined by $g$.
Then, $A$ is normal, so it is diagonalizable over $\C$ and it admits a $g$-unitary basis of eigenvectors.
Let $\mathcal{B}=\pg{X_2,\dots,X_n}$ be such a $g$-unitary basis of $\pt{\abj^\C}^{1,0}$ of eigenvectors of $A$, with respective eigenvalues $z_2,\dots,z_n\in\C$.
By Lemma \ref{lskt}, we need to prove that $\lambda A_0+2A_0^2=0$, or equivalently that $(\lambda A_0+2A_0^2)X_j=0$, for all $j=2,\dots,n$.
Recall that $2A_0=A+A^t$, so $2A_0X_j=(z_j+\bar z_j)X_j$.
Indeed, for $j,k=2,\dots,n$,
\begin{equation*}
\begin{aligned}
   g\pt{2A_0X_j,\bar X_k}
   =&g\pt{\pt{A+A^t}X_j,\bar X_k}=g\pt{AX_j,\bar X_k}+g\pt{X_j,A\bar X_k}\\
   =&\pt{z_j+\bar z_k}g\pt{X_j,\bar X_k}.
\end{aligned}\end{equation*}
Thus, for all $j=2,\dots,n$, we have
\begin{equation*}
\begin{aligned}
    2\pt{\lambda A_0+2A_0^2} X_j=&(\lambda \id +2A_0)2A_0X_j=(z_j+\bar z_j)(\lambda \id +2A_0)X_j\\
    =&(z_j+\bar z_j)(\lambda +z_j+\bar z_j)X_j=0,
\end{aligned}\end{equation*}
concluding the proof.
\end{proof}

We finally extend the characterization of pluriclosedness on an almost abelian Lie algebra with complex structure and with a given presentation, to all Hermitian metrics, not necessarily compatible with the given presentation.

\begin{proposition}\label{propSKT}
    Let $J$ be a complex structure on an almost abelian Lie algebra $(\g,\ab)$ with presentation $(\lambda,v,A)$. 
    Then, there exists a pluriclosed metric on $(\g,J)$ if and only if $A$ is diagonalizable over $\C$ and the real part of its spectrum is contained in $\pg{0,-\lambda/2}$. 
\end{proposition}

\begin{proof}
Let $(\lambda,v,A)$ be a presentation determined by $y\in\ab\setminus\abj$. Assume that $g$ is a pluriclosed metric on $(\g,J)$ and let $\tilde y\in\ab\setminus\abj$ be the vector, unique up to sign, such that $g\in\mathcal{G}(\tilde y)$. 
Then, $\tilde y=c y+a$ for some $c\in\R^*$, $a\in\abj$, and let $(\tilde\lambda,\tilde v,\tilde A)$ be defined by \eqref{changeTilde}. 
By Corollary \ref{corSKT}, $\tilde A$ is diagonalizable over $\C$ and the real part of its spectrum is contained in $\pg{0,-\tilde\lambda/2}$. 
Since $\tilde\lambda=c\lambda$ and $\tilde A=cA$, it follows that $A$ is diagonalizable over $\C$ and the real part of its spectrum is contained in $\pg{0,-\lambda/2}$. 

Conversely, assume that $A$ is diagonalizable over $\C$ and the real part of its spectrum is contained in $\pg{0,-\lambda/2}$. 
Let $\mathcal{B}=\pg{X_2,\dots,X_n}$ be a basis of $\pt{\abj^\C}^{1,0}$ of eigenvectors of $A$, with respective eigenvalues $z_2,\dots,z_n$, and let $g_0$ be the Hermitian metric on $(\abj,J)$ such that $\mathcal B$ is unitary.
Then, $A$ is normal with respect to $g_0$, so every extension of $g_0$ to a Hermitian metric on  $(\g,J)$ is pluriclosed by Corollary \ref{corSKT}.
\end{proof}

\begin{remark}
The unimodularity of $\g$ gives restrictions in Proposition \ref{propSKT} on the multiplicity of the eigenvalues with real part $-\frac\lambda2$.
More precisely, if $\g$ is unimodular and $\lambda=0$, there exists a pluriclosed metric on $(\g,J)$ if and only if $A$ is diagonalizable over $\C$ and has purely imaginary spectrum.
On the other hand, if $\g$ is unimodular and $\lambda\neq0$, there exists a pluriclosed metric on $(\g,J)$ if and only if $A$ is diagonalizable over $\C$ and all the eigenvalues of $A$ are purely imaginary except for two of them, being the conjugate one of another, and having real part $-\lambda/2.$
\end{remark}

\subsection{Gauduchon metrics}

As before, we will first give a characterization of Gauduchon metrics compatible with a given presentation, and then extend it to all Hermitian metrics on the given almost abelian Lie algebra.

\begin{lemma}\label{lGaud}
    Let $J$ be a complex structure on an almost abelian Lie algebra $(\g,\ab)$ with presentation $(\lambda,v,A)$ determined by $y\in\ab\setminus\abj$. 
    A Hermitian metric $g\in\mathcal{G}(y)$ is Gauduchon if and only if $\tr A\in\pg{0,-\lambda}$.
\end{lemma}

\begin{proof}
    Let $g\in\mathcal{G}(y)$ and $\omega=g(J\cdot,\cdot)$ be its fundamental form.
    Then, by \eqref{ddc}, 
    \begin{equation}\label{eqGaud}
        dd^c\pt{\omega^{n-1}}=dd^c\pt{\omega_0^{n-1}}=-2i\alpha\wedge\bar\alpha\wedge\pt{\lambda A^*+A^* A^*}\omega_0^{n-1},
    \end{equation}
    where $\omega_0=\rest{\omega}{\abj}\in \Lambda^{1,1}\abj$. 
    In particular, $\omega_0^{n-1}$ is a volume form on $\abj$, so that $A^*\omega_0^{n-1}=(\tr A)\omega_0^{n-1}$.

    Thus, $g$ is Gauduchon if and only if $dd^c\omega=0$, which by \eqref{eqGaud} is equivalent to 
    \begin{equation*}
        0=\pt{\lambda \id +A^* }A^*\omega_0^{n-1}=\pt{\lambda+\tr A}\pt{\tr A}\omega_0^{n-1},
    \end{equation*}
    namely $\tr A\in\pg{0,-\lambda}$.
\end{proof}

\begin{corollary}
    Let $J$ be a complex structure on an almost abelian Lie algebra $(\g,\ab)$.
If $(\g,J)$ admits a Gauduchon metric, then every Hermitian metric on $(\g,J)$ is Gauduchon.
\end{corollary}

\begin{proof}
Let $g$ be a Gauduchon metric on $(\g,J)$ and $y\in\ab\setminus\abj$ be the vector, unique up to sign, such that $g\in\mathcal{G}(y)$. 
Then, if $(\lambda,v,A)$ is the presentation determined by $y$, we know by Lemma \ref{lGaud} that $\tr A\in\pg{0,-\lambda}$.
Since this condition does not depend on $g$, but only on the presentation, it follows that every Hermitian metric in $\mathcal{G}(y)$ is Gauduchon, once again by Lemma \ref{lGaud}.

If $\tilde g$ is any other Hermitian metric, take $\tilde y$ such that $\tilde g\in\mathcal{G}(\tilde y)$, and let $c\in\R^*$ be such that $\tilde y-cy\in\abj$.
Then, if $(\tilde \lambda,\tilde v,\tilde A)$ is the presentation determined by $\tilde y$, we find by \eqref{changeTilde}:
\begin{equation*}
    \tr\tilde A=c\tr A\in\pg{0,-c\lambda}=\pg{0,-\tilde\lambda}.
\end{equation*}
Thus, $\tilde g$ is Gauduchon by Lemma \ref{lGaud}.
\end{proof}

As a consequence, we find the following characterization of the existence of Gauduchon metrics on an almost abelian Lie algebra with complex structure:

\begin{proposition}\label{410}
    Let $J$ be a complex structure on an almost abelian Lie algebra $(\g,\ab)$ with presentation $(\lambda,v,A)$. 
    Then, there exists a Gauduchon metric on $(\g,J)$ if and only if $\tr A\in\pg{0,-\lambda}$.
\end{proposition}

\subsection{Genuine \texorpdfstring{$p$}{p}-pluriclosed structures}

In this final part, we will characterize almost abelian Lie algebras of dimension $2n$ admitting $p$-pluriclosed structures for some $2\leq p\leq n-2$. The key argument is the following:

\begin{theorem}\label{411}
    Let $J$ be a complex structure on an almost abelian Lie algebra $(\g,\ab)$ of real dimension $2n\ge8$.
   If $(\g,J)$ admits a $p$-pluriclosed structure for some  $1\leq p\le n-2$, then $A$ is diagonalizable over $\C$, for every presentation $(\lambda,v,A)$ of $(\g,\ab,J)$.      
\end{theorem}

\begin{proof}
Let $\Omega$ be a $p$-pluriclosed structure on $(\g,J)$ and denote $\Omega_J\coloneqq\rest{\Omega}{\abj}$. From \eqref{ddc} and \eqref{dOmega}, the condition $dd^c\Omega=0$ is equivalent to
\begin{equation}\label{daao}
    (\lambda\id +A^*)A^*\Omega_J=0.
\end{equation}

\textit{Step 1.} Let $\pg{X_2,\dots,X_n}\in(V^\C)^{1,0}$ be a Jordan basis for $A$, satisfying \eqref{jordandual}.
Define inductively $\psi_0=\Omega_J$, $\phi_0=A^*\Omega_J$, $\xi_0=A^*A^*\Omega_J$, and $\psi_j,\phi_j,\xi_j$ as in Lemma \ref{induction}. Then, by Lemma \ref{induction},
\begin{equation*}
    A^*\psi_j=U_j\psi_j+\phi_j,\quad
    A^*\phi_j=U_j\phi_j+\xi_j.
\end{equation*}
Combining these two equations, we get 
    \begin{equation*}
    A^*A^*\psi_j=U_j^2\psi_j+2U_j\phi_j+\xi_j=U_j^2\psi_j+(2U_j-\lambda)\phi_j,
    \end{equation*}
where the last equality holds because $\lambda\phi_0+\xi_0=0$ by \eqref{daao}.
Then, $P_j(A^*)\psi_j=0,$ with 
\begin{equation*}
    P_j(s)=s^2+(\lambda-2U_j)s+U_j(U_j-\lambda)=\pt{s-U_j}\pt{s-U_j+\lambda}.
\end{equation*}
For $j=p-1$, $\psi_{p-1}$ is the fundamental form associated to a Hermitian metric on $(V_{n-p}^*,J)$, so we are in the assumptions of Proposition \ref{PAomega=0}, for $M\in\End(V_{n-p})$ defined as in \eqref{endM}, with $M^*=\rest{A^*}{V_{n-p}^*}$.
By Proposition \ref{PAomega=0}, we conclude that $M$ is diagonalizable on $V_{n-p}$, namely
\begin{equation}\label{deltaPl}
    \delta_l=0,\qquad\forall\ l=2,\dots,n-p.
    \end{equation}
Consequently, $A$ preserves $V_{n-p-1}$, its restriction is diagonalizable, and $A^*$ preserves $W_k^*$, for all $k=p,\dots,n-1$.

\textit{Step 2.}
As in the proof of \Cref{pKahler}, $W_{p+1}$ is well defined since $p\le n-2$, and $\rest{\Omega_J}{W_{p+1}}=\sigma^p$, where $\sigma$ is the fundamental form of a Hermitian metric on the $(p+1)$-dimensional complex vector space $(W_{p+1}^\R,J)$.
Since $A^*$ preserves $W_{p+1}^*$ and $A$ preserves each $W_{k}$, we get $A^*|_{W_{p+1}^*}=A_{p+1}^*$, where $A_{p+1}:=\rest{A}{W_{p+1}}$, so 
\begin{equation}\label{AsigmapPl}
\begin{aligned}
    0\,=&\rest{\pt{(\lambda \id +A^*)A^*\Omega_J}}{W_{p+1}}=(\lambda \id +A_{p+1}^*)A_{p+1}^*\pt{\rest{\Omega_J}{W_{p+1}}}\\
    =&\,\,(\lambda \id +A_{p+1}^*)A_{p+1}^*\pt{\sigma^p}=(\lambda \id +A_{p+1}^*)\pt{p\pt{A_{p+1}^*\sigma}\wedge\sigma^{p-1}}\\
    =&\,\,p\,\sigma^{p-1}\wedge{(\lambda \id +A_{p+1}^*)A_{p+1}^*\sigma}+p(p-1)\,\sigma^{p-2}\wedge\pt{A_{p+1}^*\sigma}^2,
\end{aligned}
\end{equation}
using Lemma \ref{dualEnd} for the last equality.
Let us denote by $L$ the wedge product with $\sigma$, and by $\Lambda$ the Lefschetz operator associated to $\sigma$, namely the $\sigma$-adjoint of $L$.
Then, we are under the assumptions of Lemma \ref{A1} in the Appendix, so we conclude that
\begin{eqnarray*}
0&=&L\Lambda\pt{\pt{\lambda \id  +A_{p+1}^*}A_{p+1}^*\sigma}+(p-1)\pt{\pt{\lambda \id  +A_{p+1}^*}A_{p+1}^*\sigma+\Lambda(A_{p+1}^*\sigma)^2}\\
&=&J^\flat\Lambda\pt{2\pt{\lambda A_0+2A_0^2+\pq{A_0,A_1}}J}^\flat\\
&&+(p-1)\pt{\pt{2\pt{\lambda A_0+2A_0^2+\pq{A_0,A_1}}J}^\flat+\Lambda(A_{p+1}^*\sigma)^2},
\end{eqnarray*}
where  $A_0$ and $A_1$ denote the $\sigma$-symmetric and $\sigma$-skew-symmetric parts of $A_{p+1}$.
It follows  by Lemma \ref{A2} in the Appendix, together with the fact that $\tr \pq{A_0,A_1}=0,$ that there exists a quadratic polynomial $Q$ such that 
\begin{equation*}
    Q(A_0)+[A_0,A_1]=0.
\end{equation*}
By Lemma \ref{polyn}, we conclude that $[A_0,A_1]=0$, and so $A_{p+1}$ is diagonalizable over $\C$.
Together with Step 1, and more precisely \eqref{deltaPl}, this means that $A$ is diagonalizable over $\C$, and the theorem is proved.
\end{proof}

In order to state the characterization of $p$-pluriclosed structures, let $(\lambda,v,A)$ be a presentation of an almost abelian Lie algebra $(\g,\ab)$ of real dimension $2n\ge 8$ endowed with a complex structure $J$. Denote by ${z_2,\dots,z_n}\in\C$ the eigenvalues of $A$ as an endomorphism of $(\abj^\C)^{1,0}$, and by $u_j:=z_j+\bar z_j$, for $2\le j\le n$.
Then, the following holds:

\begin{theorem}\label{AdiagPL}
     There exists a $p$-pluriclosed structure on $(\g,J)$, for some $2\leq p\leq n-2$, if and only if $A$ is diagonalizable, and:
    \begin{enumerate}
        \item[$(i)$] either $p=n-2$, and there exists $k\in\{0,\dots,n-1\}$ such that 
        \begin{equation}
            u_2=\dots=u_{k+1}=t,\quad u_{k+2}=\dots=u_n=t+\lambda, \quad \mbox{for } t:=\frac{k+1-n}{n-2}\lambda,
        \end{equation}
        \vskip0.1cm 
        \item[$(ii)$] or $2\le p\le n-3$, and, up to reordering, one of the $4$ cases below holds:
        \begin{enumerate}
            \item[$(a)$] $u_2=\ldots=u_n=0$, 
            \item[$(b)$] $u_2=\ldots=u_n=-\frac{\lambda}{p}$,
            \item[$(c)$] $u_2=\ldots =u_{n-1}=0$, $u_n=-\lambda$, in which case $(\g,J)$ admits a pluriclosed metric,
            \item[$(d)$] $u_2=\ldots =u_{n-1}=-\frac{\lambda}{p}$, $u_n=\frac{p-1}{p}\lambda$.
        \end{enumerate}
    \end{enumerate}
\end{theorem}

\begin{proof} Assume that $\Omega$ is a $p$-pluriclosed structure on $(\g,J)$ for some $2\leq p\leq n-2$. Then, $A$ is diagonalizable over $\C$ by Theorem \ref{411}. We denote by $X_j\in(\abj^\C)^{1,0}$ an eigenvector corresponding to $z_j$ for every $j\in\{2,\ldots,n\}$.
Then, for every choice of $2\le j_1<\dots<j_p\le n$, we obtain from \eqref{daao}:
\begin{equation*}
\begin{aligned}
       0=& (\lambda\id +A^*)A^*\Omega_J\pt{X_{j_1},\bar X_{j_1},\dots,X_{j_p},\bar X_{j_p}}\\
       =&\pt{\lambda+\sum_{k=1}^p u_{j_k}}\pt{\sum_{k=1}^p u_{j_k}}\Omega\pt{X_{j_1},\bar X_{j_1},\dots,X_{j_p},\bar X_{j_p}}.
\end{aligned}
\end{equation*}
Since $\Omega$ is transverse, we thus have 
\begin{equation}\label{416}
    \pt{\lambda+\sum_{j\in I} u_{j}}\pt{\sum_{j\in I} u_{j}}=0,
\end{equation}
for every multi-index $I\subset\pg{2,\dots,n}$ with $|I|=p$.
The thesis follows by Lemma \ref{Alemma} in the Appendix.

Conversely, if (i) or (ii) hold, Lemma \ref{Alemma} implies that \eqref{416} holds. Let $\omega_0$ be the fundamental 2-form of the Hermitian metric for which some basis of eigenvectors of $A$ is orthonormal. Then,  every transverse $(p,p)$-form $\Omega$ on $(\g,J)$ such that $\Omega|_{\abj}=\omega_0^p$ satisfies \eqref{daao}. Thus, by \eqref{ddc}, $dd^c\Omega=0$, so $\Omega$ is a $p$-pluriclosed structure.
\end{proof}

\begin{remark}
It follows from the proof above that if $(\g,J)$ admits a $p$-pluriclosed structure, for some $p>1$, then it also carries a Hermitian metric such that the $p$-th exterior power of its fundamental form is $dd^c$-closed.
In particular, if $(\g,J)$ admits a $(n-2)$-pluriclosed structure, then it also carries an astheno-Kähler metric.
\end{remark}

\begin{remark}\label{sum}
    In case (i) of Theorem \ref{AdiagPL} one has $u_2+\ldots+u_n=t$. Moreover, the Lie algebras corresponding to case (i) of Theorem \ref{AdiagPL} with $k=1$ are the same as those corresponding to case (iic) of Theorem \ref{AdiagPL}.
\end{remark}

\begin{remark}
    If $\lambda=0$, Theorem \ref{AdiagPL} shows that for every $p\le n-2$, $(\g,J)$ admits a $p$-pluriclosed structure if and only if $A$ is diagonalizable over $\C$ and has purely imaginary spectrum. In this case $\g$ is unimodular by Remark \ref{unimod}, and Proposition \ref{lemma17} shows that $(\g,J)$ admits a Kähler metric if and only if $v\in A(\abj)$. 
\end{remark}

\begin{remark} \label{r416}
    If $\lambda\neq 0$, among the Lie algebras occurring in Theorem \ref{AdiagPL}, the only unimodular ones are:
    \begin{itemize}
        \item (iic), which by Remark \ref{sum} is the same as (i) with $k=1$, or
        \item  (iid), when $n$ is odd and $p=\frac{n-1}{2}$.
    \end{itemize} 
\end{remark}
 We conclude this section by analyzing the existence of lattices in the simply connected Lie groups corresponding to the Lie algebras appearing in Remark \ref{r416}.
 
\begin{example}\label{ex1}
    Let $(\g,\ab)$ be the almost abelian Lie algebra with a complex structure $J$ and a presentation $(\lambda,v,A)$, with $\lambda\in\R^*$, $v=0$ and $A$ diagonalizable over $\C$, whose eigenvalues are in case (i), with $k=1$ (or (iic)) of Theorem \ref{AdiagPL}.
    Then, by Propositions \ref{propSKT} and \ref{410} and Theorem \ref{AdiagPL}, $(\g,J)$ admits $p$-pluriclosed structures, for $1\le p\le n-1$.
    By assumption, there is a basis $\mathcal{B}_\varphi$ of $\ab$ such that the matrix associated to the endomorphism $\varphi$ defining the structure of semidirect product of $\g$ is 
    \begin{equation}\label{matrixB}
        B:=\pt{\begin{array}{c|cccc}
        \lambda &       & &0 &   \\ \hline
                & A_2   &&&\\
         v       &       &A_3 &&\\
               &      & &\ddots &\\
                &&&&A_n
        \end{array}},
    \end{equation}
    where each $A_j$, $j=2,\dots,n$ is a $2\times2$ block of the following form:
    \begin{equation*}
        A_2=
        \begin{pmatrix}
            -\frac\lambda2  &   -b_2\\
            b_2            &   -\frac\lambda2
        \end{pmatrix},
        \quad\quad
        A_j=
        \begin{pmatrix}
            0   &   -b_j\\
            b_j &   0
        \end{pmatrix},
        \quad\quad
        j=3,\dots,n,
    \end{equation*}
  for some  $b_2,\dots,b_n\in\R$.
  We want to show that the $b_j$'s and $v$ can be chosen such that the corresponding simply connected Lie group admits a cocompact lattice.
By a result of Bock \cite{Bo16}, a unimodular almost abelian Lie group $G=\R\ltimes_\rho \R^{2n-1}$ with Lie algebra $\g=\R\ltimes_\varphi\ab$ admits a lattice if and only if there is a basis of $\ab=\R^{2n-1}$ and $s\in\R^*$ such that the matrix of $\rho(s)=e^{s\varphi}$ in that basis is in $\mathrm{SL}(2n-1,\Z)$.

Let $m,l_2,\dots,l_n\in\Z$ be integers with $m>0$.
The polynomial $\mathrm{X}^3+m\mathrm{X}-1$ has $2$ complex conjugate roots $e^{-r+i\theta},e^{-r-i\theta}$, $\theta\neq0$, and one real root, $e^{2r} $.
Thus its companion matrix is diagonalizable:
\begin{equation}\label{eqmatrix}
    \begin{pmatrix}
        0&0&1\\
        1&0&-m\\
        0&1&0
    \end{pmatrix}
    \sim
    \begin{pmatrix}
        e^{2r}&0&0\\
    0&e^{-r+i\theta}&0\\
    0&0&e^{-r-i\theta}
    \end{pmatrix}=:C.
\end{equation}
We set 
\begin{equation*}
s:=\frac{2r}\lambda  ,\qquad
    b_2:=\frac\theta s=\frac{\theta\lambda}{2r},\qquad
    b_j:=\frac{2\pi l_j}s=\frac{\pi\lambda l_j}{r}.
\end{equation*}
Then, $e^{s\lambda}=e^{2r}$, $e^{sA_j}=I_2$, for all $j=3,\dots,n,$ and 
\begin{equation*}
    e^{sA_2}
    =e^{-\frac{s\lambda}{2}}\begin{pmatrix}
    \cos(sb_2)&-\sin(sb_2)\\
    \sin(sb_2)&\cos(sb_2)
\end{pmatrix}
    =e^{-r}\begin{pmatrix}
    \cos\theta&-\sin\theta\\
    \sin\theta&\cos\theta
\end{pmatrix}
\sim
\begin{pmatrix}
    e^{-r+i\theta}&0\\
    0&e^{-r-i\theta}
\end{pmatrix}.
\end{equation*}
In particular, if $v=0$, the matrix of $e^{s\varphi}$ in the basis $\mathcal{B}_\varphi$ is
\begin{equation*}
    e^{sB}\sim
    \begin{pmatrix}
        C &0\\
        0&I_{2n-4}
    \end{pmatrix},
\end{equation*}
which by \eqref{eqmatrix} is conjugate to a matrix in $\mathrm{SL}(2n-1,\Z)$. Note that if the $l_j$'s are non-zero, the Lie algebra $\g$ does not have abelian factors.
\end{example}

\begin{example}
    Consider now a unimodular almost abelian Lie algebra $\g=\R\ltimes_\varphi\ab$ corresponding to case (iid) of Theorem \ref{AdiagPL}, with a presentation $(\lambda,v,A)$, satisfying $\lambda\neq0$ and $A$ diagonalizable over $\C$.
    By Remark \ref{r416}, $n:=\dim_\C\g=2p+1$ is odd, with $p\ge2$.
The matrix of $\varphi$ with respect to some basis of $\ab$ can be written as in \eqref{matrixB}, with
    \begin{equation*}
        A_2=
        \begin{pmatrix}
            \frac{p-1}{2p}\lambda  &   -b_2\\
            b_2            &   \frac{p-1}{2p}\lambda
        \end{pmatrix},
        \quad\quad
        A_j=
        \begin{pmatrix}
            -\frac{\lambda}{2p}   &   -b_j\\
            b_j &   -\frac{\lambda}{2p}
        \end{pmatrix},
        \quad\quad
        j=3,\dots,n,
    \end{equation*}
  for some  $b_2,\dots,b_n\in\R$.
  By Bock's criterion, a necessary condition for the  corresponding simply connected Lie group to admit a cocompact lattice is that the characteristic polynomial $\chi$ of $e^{s\varphi}$ has integer coefficients, for some $s\in\R^*$.
  Using the notation $\mu:=e^{-s\frac{\lambda}{2p}}\in(0,\infty)\setminus\pg1$, the moduli of the roots of $\chi$ are:
  \begin{itemize}
      \item  $\mu$, with multiplicity $4p-2$,
      \item $\mu^{1-p}$, with multiplicity $2$,
      \item $\mu^{-2p}$, with multiplicity $1$.
  \end{itemize}
  For $p=1$, we recover Example \ref{ex1}, in complex dimension $3$.
  For  $p\ge2$, it is an open problem to determine whether such a polynomial exists. 
\end{example}

\section{LCK and Bismut-Ricci flat metrics}\label{secLCK}

In this section we characterize almost abelian Lie algebras with complex structures $(\g,J)$ carrying LCK and Bismut-Ricci flat metrics, in terms of a given presentation.

\subsection{LCK metrics}
We recall the following result (first obtained in \cite{AO18}): 
\begin{lemma}\label{lck}
    On an almost abelian Lie algebra $(\g,\ab)$ of real dimension $2n$ with complex structure $J$ and presentation $(\lambda,v,A)$ determined by $y\in\ab\setminus\abj$, a Hermitian metric $g\in\mathcal{G}(y)$ is LCK if and only if 
    \begin{itemize}
        \item either $v=0$ and $A=\mu\id+A_1$, for some $\mu\in\R$ and some matrix $A_1$ skew-symmetric with respect to $g|_{\abj}$,
        \item or $n=2$ and $A=0$.
    \end{itemize} 
\end{lemma}
\begin{proof}
The fundamental form of $g$, defined by $\omega:=g(J\cdot,\cdot)$ can be  written as
    \[\omega=x^\flat\wedge y^\flat+\omega_0,\] 
    where  $\omega_0\in \Lambda^{1,1}\abj$ is  the fundamental form of $g|_{\abj}$. The structure $(g,J)$ is LCK if and only if there exists a closed 1-form $\theta\in\g^*$ such that $d\omega=\theta\wedge\omega$.
  By Corollary \ref{c13}, we have
    \[
d\omega=d\omega_0=-x^\flat\wedge y^\flat\wedge\iota_{v}\omega_0-x^\flat\wedge A^*\omega_0,
    \]
    with $x:=-Jy$.
    Then, expressing $\theta=c_xx^\flat+c_yy^\flat+\theta_0$, with $c_x,c_y\in\R$ and $\theta_0\in\abj^*$, we get 
    \[
    0=d\theta=-(c_y\lambda +\iota_v\theta_0)x^\flat\wedge y^\flat-x^\flat\wedge A^*\theta_0,\] and
    \[
    \theta\wedge\omega=x^\flat\wedge y^\flat\wedge\theta_0+\theta_0\wedge\omega_0+c_xx^\flat\wedge\omega_0+c_yy^\flat\wedge\omega_0.
    \]
    Comparing types, this shows that $(g,J)$ is LCK if and only if there exists a solution $(c_x,x_y,\theta_0)$ to the following system:
    \begin{equation}\label{elck}
\iota_{v}\omega_0=-\theta_0,\qquad A^*\omega_0=-c_x\omega_0,\qquad \theta_0\wedge\omega_0=0,\qquad c_y=0,\qquad A^*\theta_0=0.
    \end{equation} 
    Let us write $A=A_0+A_1$ with $A_0$ symmetric and $A_1$ skew-symmetric with respect to $g|_{\abj}$. By \eqref{ao}, the second equation in \eqref{elck} is equivalent to $A_0=-\frac{c_x}2\id$. 
    We distinguish two cases:
    
    If $v=0$, the first equation in \eqref{elck} is equivalent to $\theta_0=0$, and thus \eqref{elck} has a solution if and only if we are in the first case of the lemma. 
       
        If $v\neq 0$, and \eqref{elck} has a solution, the first equation shows that $\theta_0\neq 0$, so the third equation implies $n=2$, and the last equation shows that $A$ is non-invertible. Moreover, since $A=-\frac{c_x}2\id+A_1$ is a $2\times 2$ matrix, it is either $0$ or invertible, so the only possibility is $A=0$, and thus we are in the second case of the lemma. Conversely, if $n=2$ and $A=0$, the system \eqref{elck} has the solution $c_x=c_y=0,\ \theta_0=-\iota_v\omega_0$.
\end{proof}

We will now give a general criterion for an almost abelian Lie algebra $(\g,\ab)$ with complex structure $J$ to admit LCK metrics in terms of a presentation $(\lambda,v,A)$. 

\begin{proposition}\label{p53}
    Let $J$ be a complex structure on an almost abelian Lie algebra $(\g,\ab)$ with presentation $(\lambda,v,A)$ determined by some $y\in\ab\setminus\abj$. Then there exists an LCK metric on $(\g,J)$ if and only if:
    \begin{itemize}
        \item either $v\in (A-\lambda \id )(\abj)$ and $A$ is diagonalizable over $\C$, with all eigenvalues having the same real part, 
        \item or $n=2$ and $A=0$.
    \end{itemize} 
\end{proposition}

\begin{proof}
Assume first that $g$ is an LCK metric on $(\g,J)$ and let $\tilde y\in\ab\setminus\abj$ be the vector, unique up to sign, such that $g\in\mathcal{G}(\tilde y)$. 
Write $\tilde y=c y+a$ for some non-zero $c\in\R^*$ and $a\in\abj$, and let $(\tilde\lambda,\tilde v,\tilde A)$ be defined by \eqref{changeTilde}. 
By Lemma \ref{lck}, since $g$ is LCK, we either have $\tilde v=0$ and $\tilde A=\tilde\mu\id+\tilde A_1$, with $\tilde A_1$ skew-symmetric with respect to $g|_{\abj}$, or $n=2$ and $\tilde A=0$. 
In the first case, by \eqref{changeTilde}, this means that $\tilde v =cv+(A-\lambda \id )a=0$, so $v\in (A-\lambda \id )(\abj)$, and  $A=\mu\id+A_1$, where $\mu:=\frac1c\tilde\mu$, and $A_1:=\frac1c\tilde A_1$ is skew-symmetric as well with respect to $g|_{\abj}$. 
In particular, $A$ is diagonalizable over $\C$ and each eigenvalue has the same real part $\mu$. In the second case, we obtain directly $n=2$ and $A=0$.

Conversely, assume that either $v\in (A-\lambda \id )(\abj)$ and $A$ is diagonalizable over $\C$, with all eigenvalues having the same real part $\mu$, or $n=2$ and $A=0$. In the second case the conclusion follows from the second item of Lemma \ref{lck}. In the first case, the assumption on $A$ ensures the existence of a Hermitian metric $g_0$ on $\abj$ such that $A-\mu\id$ is skew-symmetric with respect to $g_0$, as well as the existence of $a\in \abj$ such that $v+(A-\lambda \id )a=0$.
Then, the vector $\tilde y\in\ab\setminus\abj$ defined by $\tilde y:=y+a$ determines a unique Hermitian metric $g\in\mathcal{G}(\tilde y)$ on $(\g,J)$ with $\rest{g}{\abj}=g_0$, and $g$ is LCK by the first item of Lemma \ref{lck}.
\end{proof}

\subsection{Bismut-Ricci flat metrics}
Let $(\g,J)$ be an almost abelian Lie algebra with complex structure, with presentation $(\lambda,v,A)$ determined by $y$. 
    Let $g$ be a Hermitian metric in $\mathcal{G}(y)$, and denote with $\rho^B$ and $\rho^C$ the Ricci forms of the Bismut and Chern connections of $(J,g)$ (see for instance \cite{fp23}), also called Bismut-Ricci form and Chern-Ricci form, respectively.
By \cite[Lemma 6.1]{LR}, the Chern-Ricci form can be written in terms of the presentation $(\lambda,v,A)$ as
\begin{equation*}
    \rho^C=-\frac\lambda 2(2\lambda+\tr A)\,x^\flat\wedge y^\flat,
\end{equation*}
where $x,y,\lambda$ in our notation correspond respectively to $e^{2n},-e^1,c$ in \cite{LR}.
\begin{remark}
We note that, by Corollary \ref{dkformJ}, $\lambda x^\flat\wedge y^\flat$ is exact, so $\rho^C$ is exact as well.
Consequently, the first Chern class of every left-invariant complex structure on an almost abelian solvmanifold is vanishing.
\end{remark}

On the other hand, by \cite[Equation (2.4)]{fp23},
\begin{equation}\label{BRform}
    \rho^B=-\pt{\lambda^2-\frac\lambda2\tr A+\lVert v\rVert^2_g}\,x^\flat\wedge y^\flat-x^\flat\wedge A^*v^\flat,
\end{equation}
where $A^*v^\flat=(A^tv)^\flat$, denoting with $\cdot^t$ transposition with respect to $g$, and $x,y,\lambda,v$ in our notation correspond respectively to $e^{2n},-e^1,a,-v$ in \cite{fp23}.

A direct consequence of \eqref{BRform} is that a necessary condition for an almost abelian Lie algebra with complex structure $(\g,J)$ to admit a Bismut-Ricci flat metric is that
    \begin{equation*}
\lambda\pt{2\lambda-\tr A}  \le0. 
    \end{equation*}
In the next proposition, we fully characterize the existence of Bismut-Ricci flat metrics, in terms of the presentation.

\begin{proposition}\label{pCYT}
    Let $J$ be a complex structure on an almost abelian Lie algebra $(\g,\ab)$ with presentation $(\lambda,v,A)$. 
    Then, $(\g,J)$ admits a Bismut-Ricci flat metric if and only if one of the following holds:
    \begin{enumerate}[label=(\roman*)]
        \item\label{BRf1}  $\lambda\pt{2\lambda-\tr A}=0$ and $v$ is in the image of $(A-\lambda \id)$, or
        \item\label{BRf2}  $\lambda\pt{2\lambda-\tr A}<0$ and  $v\in(A-\lambda \id)(\abj)+(\abj\setminus A(\abj))$.
        In particular, in this case $\ker A\neq \pg0$.
    \end{enumerate}
\end{proposition}

\begin{proof}
Let $g$ be a Bismut-Ricci flat metric on $(\g,J)$, and let $\tilde y\in\ab\setminus\abj$ such that $g\in\mathcal{G}(\tilde y)$.
Write $\tilde y=c y+a$ for some $c\in\R^*$ and $a\in\abj$, and let $(\tilde\lambda,\tilde v,\tilde A)$ be defined by \eqref{changeTilde}. 
    Then, by \eqref{BRform}, we get
\begin{equation}\label{brfeq}
    \norm{\tilde v}^2_g=-\tilde\lambda\pt{\tilde\lambda-\frac{\tr\tilde A}{2}} ,
    \qquad
    \tilde A^t\tilde v=0.
\end{equation}
Since $\tilde A=cA,$ the second condition together with \eqref{changeTilde} reads
\begin{equation}\label{vImA}
    \tilde v = c(cv+(A-\lambda \id )a)\in\ker A^t=A(\abj)^{\perp_g}.
\end{equation}
Moreover, the first equality in \eqref{brfeq} yields $0\ge \tilde\lambda\pt{2\tilde\lambda-\tr\tilde A}=c^2\lambda\pt{2\lambda-\tr A}$.
If  $\lambda\pt{2\lambda-\tr A}=0$, then $\tilde v=0$, so $v$ is in the image of $(A-\lambda \id)$.
If $\lambda\pt{2\lambda-\tr A}<0$, by the first equation in \eqref{brfeq} we get $\tilde v\neq 0$, and since $A(\abj)^{\perp_g}\setminus\pg0\subset \abj\setminus A(\abj),$ we obtain from \eqref{vImA}
\begin{equation*}
    v=-(A-\lambda \id )\frac{a}{c}+\frac{\tilde v}{c^2}
    \in
    (A-\lambda \id)(\abj)+(\abj\setminus A(\abj)).
\end{equation*}
Note that, in this case, since $\tilde v$ is a non-zero vector orthogonal to the image of $A$, then $\ker A\neq\pg0$.

For the converse, let us deal with case \ref{BRf1} first.
Let $a\in\abj$ such that $v=-(A-\lambda \id )a$, and choose $\tilde y:=y+a.$ 
Then, the presentation  $(\tilde\lambda,\tilde v,\tilde A)$ determined by $\tilde y$ has $\tilde v=0$, by \eqref{changeTilde}.
Since $\tilde\lambda=\lambda$ and $\tilde A=A$, we also have $\tilde\lambda\pt{2\tilde\lambda-\tr\tilde A}=0$, so every Hermitian metric in $\mathcal{G}(\tilde y)$ is Bismut-Ricci flat, by \eqref{BRform}.

Assume now that item \ref{BRf2} holds, and let $v=w-(A-\lambda \id )a$, for some $a\in\abj$ and $w\in\abj\setminus A(\abj)$.
Furthermore, since $-\lambda\pt{2\lambda-\tr A}>0$ by \ref{BRf2}, there exists a Hermitian metric $g_0$ on $\abj$ such that 
\begin{equation*}
    g_0(w,A(\abj))=0,
    \qquad
    \norm{w}^2_{g_0}=-\lambda\pt{\lambda-\frac{\tr A}{2}}.
\end{equation*}
Let $\tilde y=y+a$, with presentation  $(\tilde\lambda,\tilde v,\tilde A)$ defined by \eqref{changeTilde}.
Then, $\tilde\lambda=\lambda$, $\tilde A=A$ and $\tilde v=v+(A-\lambda \id )a=w$.
It follows that the unique Hermitian metric in $\mathcal{G}(\tilde y)$ with $g|_{\abj}=g_0$ is Bismut-Ricci flat, by \eqref{BRform}.
\end{proof}

\appendix
\section{Technical results}\label{secApp}

This final section is devoted to the proof of some technical results needed in the main body of the paper. 

\subsection{Commutation relations}

We will prove here the commutation relations used in Proposition \ref{PAomega=0}.
\begin{lemma}
Let $V$ be a real vector space of dimension $2(n-1)$ endowed with a complex structure $J$, and $A\in\End(V)$, commuting with $J$.
Denote with $A^*$ be the extension of $A$ to $V^*$ given by \eqref{a*} and to $\Lambda V^*$ as a derivation, and fix a  Jordan basis $\pg{X_2,\dots,X_{n}}$ of $(V^\C)^{1,0}$ for $A$, satisfying \eqref{jordandual}. Define $\iota_k:=\iota_{X_k}$ and $\bar\iota_k:=\iota_{\bar X_k}$ for $k=2,\dots,n$.
Then, for all $j,k=2,\dots,n$, the following relations hold:
\begin{equation}\label{iotajkA}
\pq{\iota_j\bar \iota_k,A^*}=\iota_j\bar \iota_k\pt{z_j+\bar z_k}+\delta_j\iota_{j+1}\bar \iota_k+\delta_k\iota_j\bar \iota_{k+1}.
\end{equation}
Moreover, on the space of $2$-forms $\Lambda^2V^*$ we have:
\begin{equation}\label{iotajkAA2}
\begin{aligned}
\pq{\iota_j\bar \iota_k,A^*A^*}=&\pt{z_j+\bar z_k}^2\iota_j\bar \iota_k
   +2\pt{z_j+\bar z_k}\pt{\delta_j\iota_{j+1}\bar \iota_k+\delta_k\iota_j\bar \iota_{k+1}}\\
    &+2\delta_j\delta_k \iota_{j+1}\bar \iota_{k+1}+\delta_j\delta_{j+1}\iota_{j+2}\bar \iota_{k}+\delta_k\delta_{k+1}\iota_{j}\bar \iota_{k+2}.
\end{aligned}
\end{equation}
\end{lemma}

\begin{proof}
By using \eqref{iotajA} twice, we get:
    \begin{equation*}
        \pq{\iota_j\bar \iota_k,A^*}=\iota_j\pq{\bar \iota_k,A^*}+\pq{\iota_j,A^*}\bar \iota_k=\iota_j\pt{\bar z_k\bar\iota_k+\delta_k\bar \iota_{k+1}}+\pt{z_j \iota_j+\delta_j\iota_{j+1}}\bar\iota_{k},
    \end{equation*}
which proves \eqref{iotajkA}.
Similarly, we can write
\begin{equation}\label{jkAA}
\begin{aligned}
    \pq{\iota_j\bar \iota_k,A^*A^*}
    =\pq{\iota_j\bar \iota_k,A^*}A^*+A^*\pq{\iota_j\bar \iota_k,A^*}
    &=2\pq{\iota_j\bar \iota_k,A^*}A^*+\pq{A^*,\pq{\iota_j\bar \iota_k,A^*}}\\
    &=\pq{\pq{\iota_j\bar \iota_k,A^*},A^*}+2A^*\pq{\iota_j\bar \iota_k,A^*}.
\end{aligned}
\end{equation}
Applying \eqref{iotajkA} twice, we obtain:
\begin{eqnarray*}
    \pq{\pq{\iota_j\bar \iota_k,A^*},A^*}
    &=&\pt{z_j+\bar z_k}\pq{\iota_j\bar \iota_k,A^*}+\delta_j\pq{\iota_{j+1}\bar \iota_k,A^*}+\delta_k\pq{\iota_{j}\bar \iota_{k+1},A^*}\\
    &=&\pt{z_j+\bar z_k}^2\iota_j\bar \iota_k+\delta_j\pt{z_j+\bar z_k}\iota_{j+1}\bar \iota_k+\delta_k\pt{z_j+\bar z_k}\iota_j\bar \iota_{k+1}\\
    &&+\delta_j\pt{\pt{z_{j+1}+\bar z_k}\iota_{j+1}\bar \iota_k+\delta_{j+1}\iota_{j+2}\bar \iota_{k}+\delta_{k}\iota_{j+1}\bar \iota_{k+1}}\\
    &&+\delta_k\pt{\pt{z_{j}+\bar z_{k+1}}\iota_{j}\bar \iota_{k+1}+\delta_{j}\iota_{j+1}\bar \iota_{k+1}+\delta_{k+1}\iota_{j}\bar \iota_{k+2}}\\
    &=&\pt{z_j+\bar z_k}^2\iota_j\bar \iota_k
    +2\delta_j\pt{z_j+\bar z_k}\iota_{j+1}\bar \iota_k
    +2\delta_k\pt{z_j+\bar z_k}\iota_j\bar \iota_{k+1}\\
    &&+2\delta_j\delta_k \iota_{j+1}\bar \iota_{k+1}+\delta_j\delta_{j+1}\iota_{j+2}\bar \iota_{k}+\delta_k\delta_{k+1}\iota_{j}\bar \iota_{k+2},
\end{eqnarray*}
where in the last equality we used that $\delta_j z_j=\delta_jz_{j+1}$ by \eqref{jordandual}, for all $j=2,\dots,n-1$.
If $\tau$ is a $2$-form, then $\pq{\iota_j\bar \iota_k,A^*}\tau$ is a constant, and since $A^*$ acts trivially on $\Lambda^0V^*=\R$, from the second line of \eqref{jkAA} we obtain \eqref{iotajkAA2}.
\end{proof}

\subsection{The Lefschetz operator}

On a real vector space $W$ endowed with a Hermitian structure $(J,h,\sigma)$, we denote with $L=L_\sigma$ the wedge product by $\sigma$, and by $\Lambda=\Lambda_\sigma$ its dual, called the Lefschetz operator.

\begin{lemma}\label{A1}
Let $W$ be a real vector space of dimension $2(p+1)$, endowed with a Hermitian structure $(J,h,\sigma)$.
Then, for every $\beta\in\Lambda^{1,1}W^*$ and $\nu\in\Lambda^{2,2}W^*$, if 
\begin{equation*}
    L^{p-1}\beta+(p-1)L^{p-2}\nu=0,
\end{equation*}
then 
\begin{equation}
    L\Lambda\beta+(p-1)(\beta+\Lambda\nu)=0.
\end{equation}
\end{lemma}

\begin{proof}
    We recall that $\Lambda^{2,2}W^*=\Lambda_0^{2,2}\oplus L\pt{\Lambda^{1,1}W^*}$, where
    \begin{equation*}
        \Lambda_0^{2,2}:=\ker(\Lambda)\cap\Lambda^{2,2}W^*=\ker\pt{L^{p-2}}.
    \end{equation*}
    Then, we can write $(p-1)\nu=\nu_0+L\nu_1$, $\nu_0\in\Lambda_0^{2,2}$, $\nu_1\in\Lambda^{1,1}W^*$, and 
    \begin{equation*}
        0= L^{p-1}\beta+(p-1)L^{p-2}\nu
         = L^{p-1}\beta+L^{p-2}(\nu_0+L\nu_1)
         = L^{p-1}\pt{\beta+\nu_1}.
    \end{equation*}
    Since the restriction of $L^{p-1}$ to the space of $2$-forms is an isomorphism with the space of $2p$-forms, it follows that $\beta+\nu_1=0$, so
    \begin{equation*}
        0=\Lambda L(\beta+\nu_1)=\Lambda(L\beta+(p-1)\nu).
    \end{equation*}
    Using the commutation relation between $\Lambda$ and $L$, given by 
\begin{equation*}
    \pq{\Lambda,L}|_{\Lambda^{2}W^*}=(p-1)\id ,
\end{equation*}
the statement follows.
\end{proof}

We will now obtain the result about Lefschetz operators needed in the proof of Theorem \ref{411}.
\begin{lemma}\label{A2}
    Let $W$ be a real vector space endowed with a Hermitian structure $(J,h)$, and $M$ an endomorphism of $W$ commuting with $J$.
 If $M_0$ is the $h$-symmetric component of $M$, then 
 \begin{enumerate}
        \item $\Lambda\pt{2M_0J}^\flat=\tr M$.
        \item\label{LMsigma} $\Lambda\pt{\pt{M^*\sigma}^2}=4\pt{\pt{M_0\tr M_0-2M_0^2}J}^\flat$,
    \end{enumerate}
    where $\cdot^\flat$ denotes the $h$-dual of an endomorphism.
\end{lemma}

\begin{proof}
    Assume that $W$ has real dimension $2p+2$, and let $\pg{e_1,\dots,e_{2p+2}}$ be a $h$-orthonormal basis of $W$.
    Then,
    \begin{equation}
        \Lambda=\frac12\sum_{j=1}^{2p+2} \iota_{Je_j}\iota_{e_j},
    \end{equation}
    so that
    \begin{eqnarray*}
        \Lambda\pt{2M_0J}^\flat
        &=&\frac12\sum_{j=1}^{2p+2} \iota_{Je_j}\iota_{e_j}\pt{2M_0J}^\flat
        =\sum_{j=1}^{2p+2} \pt{JM_0}^\flat(e_j,Je_j)\\
        &=&\sum_{j=1}^{2p+2} h(JM_0e_j,Je_j)
        =\sum_{j=1}^{2p+2} h(M_0e_j,e_j)=\tr M_0=\tr M,
    \end{eqnarray*}
proving the first part of the statement.

For the second part, we recall that by Lemma \ref{dualEnd}, $M^*\sigma=\pt{2M_0J}^\flat$, so we get
\begin{equation}\label{Msigma}
\begin{aligned}
    \Lambda\pt{\pt{M^*\sigma}^2}
    =&\,\,4\Lambda\pt{\pt{JM_0}^\flat\wedge\pt{JM_0}^\flat}
    =2\sum_{j=1}^{2p+2} \iota_{Je_j}\iota_{e_j}\pt{\pt{JM_0}^\flat\wedge\pt{JM_0}^\flat}\\
    =&\,\,
    2\sum_{j=1}^{2p+2} \iota_{Je_j}\pt{2\,\pt{JM_0}^\flat\wedge\iota_{e_j}\pt{JM_0}^\flat}
    \\
    =&\,\,
    4\sum_{j=1}^{2p+2} \pt{\iota_{Je_j}\pt{JM_0}^\flat\wedge\iota_{e_j}\pt{JM_0}^\flat+\pt{JM_0}^\flat\iota_{Je_j}\iota_{e_j}\pt{JM_0}^\flat}
    \\
    =&\,\,
    4\pt{\sum_{j=1}^{2p+2} \pt{JM_0Je_j}^\flat\wedge\pt{JM_0e_j}^\flat+\pt{JM_0}^\flat\sum_{j=1}^{2p+2}\iota_{Je_j}\iota_{e_j}\pt{JM_0}^\flat}
    \\
    =&\,\,
    4 \pt{-\sum_{j=1}^{2p+2} \pt{M_0e_j}^\flat\wedge\pt{JM_0e_j}^\flat+\tr M_0\pt{JM_0}^\flat}.
\end{aligned}
\end{equation}
It remains to compute the first part of the last line.
For this, we fix any $w\in W$, and we have
\begin{eqnarray*}
    \iota_w\pt{\pt{M_0e_j}^\flat\wedge\pt{JM_0e_j}^\flat}
    &=&\pt{\iota_w\pt{M_0e_j}^\flat}\pt{JM_0e_j}^\flat-\pt{M_0e_j}^\flat\,\iota_w\pt{JM_0e_j}^\flat\\
    &=&h(M_0e_j,w)\pt{JM_0e_j}^\flat-\pt{M_0e_j}^\flat\, h(JM_0e_j,w),
\end{eqnarray*}
so that
\begin{eqnarray*}
    \sum_{j=1}^{2p+2}\iota_w\pt{\pt{M_0e_j}^\flat\wedge\pt{JM_0e_j}^\flat}
    &=&\pt{\sum_{j=1}^{2p+2}h(M_0e_j,w)JM_0e_j-{M_0e_j}h(JM_0e_j,w)}^\flat\\
    &=&\pt{\sum_{j=1}^{2p+2}h(e_j,M_0w)JM_0e_j+{M_0e_j}h(e_j,M_0Jw)}^\flat\\
    &=&\pt{\pt{JM_0^2+M_0^2J}w}^\flat=\iota_w\pt{2\pt{M_0^2J}^\flat},
\end{eqnarray*}
where in the last equality we used that $M_0$ commutes with $J$.
Combining this with the last line in \eqref{Msigma}, we obtain the second part of the statement.
\end{proof}

\subsection{A quadratic system}
We conclude the section computing the solutions to the quadratic system, as in the statement of \Cref{AdiagPL}.

\begin{lemma}\label{Alemma}
    Let $n,p$ be integers with $2\le p\le n-2$, and let $\lambda,u_2,\dots,u_n$ be real numbers. Then
\begin{equation}\label{realsum}
    \pt{\lambda+\sum_{j\in I} u_{j}}\pt{\sum_{j\in I} u_{j}}=0,\quad\quad\text{for all }I\subset\pg{2,\dots,n}\text{ with }|I|=p
\end{equation}
if and only if, up to rearranging the $u_j$'s, one of the following holds:
    \begin{enumerate}
        \item[$(i)$] either $p=n-2$, and there exists $k\in\{0,\dots,n-1\}$ such that 
        \begin{equation}\label{Ap=n-2}
            u_2=\dots=u_{k+1}=t,\quad u_{k+2}=\dots=u_n=t+\lambda, \quad \mbox{for } t:=\frac{k+1-n}{n-2}\lambda,
        \end{equation}
        \vskip0.1cm
        \item[$(ii)$]\label{Ap2} or $2\le p\le n-3$, and one of the cases below holds:
        \begin{enumerate}
            \item[$(a)$]\label{Aca} $u_2=\dots=u_n=0$,
            \item[$(b)$]\label{Acb} $u_2=\dots=u_n=-\frac{\lambda}{p}$,
            \item[$(c)$]\label{Acc} $u_2=\dots=u_{n-1}=0$, $u_n=-\lambda$,
            \item[$(d)$]\label{Acd} $u_2=\dots=u_{n-1}=-\frac{\lambda}{p}$, $u_n=\frac{p-1}{p}\lambda$.
        \end{enumerate}
\end{enumerate}
\end{lemma}

\begin{proof} If $\lambda=0$, then (i) or (ii) imply that $u_2=\ldots=u_n=0$, so \eqref{realsum} holds. Conversely, \eqref{realsum} implies that 
\begin{equation}\label{realsum0}
   \sum_{j\in I} u_{j}=0,\quad\quad\text{for all }I\subset\pg{2,\dots,n}\text{ with }|I|=p.
\end{equation}
 For every pair of indices $l\neq m$, there is a set $I_0\subset\{2,\ldots,n\}$ with $|I_0|=p-1$, such that $l,m\notin I_0$. Then \eqref{realsum0} applied to $I=I_0\cup\{l\}$ and $I=I_0\cup\{m\}$ implies that the $u_j$'s are all equal, and thus they all vanish by using again \eqref{realsum0}. 
 
 We now assume for the rest of the proof that $\lambda\neq0$ and consider separately the cases $p=n-2$ and $p\le n-3$.
 
(i) If $p=n-2$, let $t:=u_2+\dots+u_n$.
Then, \eqref{realsum} is equivalent to 
\begin{equation}\label{ltu}
    \pt{\lambda+t-u_j}(t-u_j)=0,
\end{equation}
for all $j=2,\dots,n$.
In other words, every $u_j$ is either $t$ or $t+\lambda.$
Let $k$ be the number of $u_j$'s equal to $t$, and rearrange the indices so that $u_2=\dots=u_{k+1}=t$, and  $u_{k+2}=\dots=u_n=t+\lambda $. Then
\begin{equation*}
    t=u_2+\dots+u_n=kt+(n-k-1)(t+\lambda)=(n-1)t+(n-k-1)\lambda,
\end{equation*}
and \eqref{Ap=n-2} follows. Conversely, \eqref{Ap=n-2} obviously implies \eqref{ltu}.

(ii) If $p\le n-3$, for every pair of indices $l\neq m$, there is a set $I_0\subset\{2,\ldots,n\}$ with $|I_0|=p-1$, such that $l,m\notin I_0$.
We denote 
\begin{equation*}
  I_l:=I_0\cup\pg l,\qquad I_m:=I_0\cup\pg m,\qquad  s\coloneqq\sum_{j\in I_0} u_{j}.
\end{equation*}
Assume that \eqref{realsum} holds. For $I=I_m$, we obtain $(\lambda+s+u_m)(s+u_m)=0$, namely either $s=-u_m$, or $s=-u_m-\lambda$.
Then, using now \eqref{realsum} with $I=I_l$, we obtain
\begin{equation*}
0=(\lambda+s+u_l)(s+u_l)=
\begin{cases}
    (\lambda-u_m+u_l)(-u_m+u_l),     &     \text{if }s=-u_m,  \\
    (-u_m+u_l)(-\lambda-u_m+u_l),     &     \text{if }s=-u_m -\lambda.
\end{cases}
\end{equation*}
The difference $u_m-u_l$ is thus either $0$, or $\pm\lambda.$ 
Since this holds for every choice of $m\neq l$, it follows that the $u_j$'s can be rearranged so that there exist $k\in\pg{0,\dots,n-1},$ $a\in\R$ with
\begin{equation}\label{kalambda}
u_2=\dots=u_{k+1}=a,\quad u_{k+2}=\dots=u_n=a+\lambda.
\end{equation}

We claim that $k\in\pg{0,1,n-2,n-1}$.
This always holds if $n=4$.
If $n\ge5$, assume for a contradiction that $2\le k\le n-3$. Then, by \eqref{kalambda}, $u_2=u_3=a$ and $u_{n-1}=u_n=a+\lambda.$
Let $I\subset\{2,\ldots,n\}$ be a set of cardinality $p$ such that $1,2\in I$, and $n-1,n\notin I$. Such a set always exists, because $2\le p\le n-3$.
Then, by \eqref{realsum}, 
\begin{equation}\label{sumI}
    \sum_{j\in I}u_j\in\pg{0,-\lambda}.
\end{equation}
Consider now $\tilde I\coloneqq( I\setminus\pg{1,2})\cup\pg{n-1,n}$, also of cardinality $p$.
Then, we compute 
\begin{equation*}
    \sum_{j\in\tilde I}u_j=\sum_{j\in I}u_j-u_1-u_2+u_{n-1}+u_n=\sum_{j\in I}u_j+2\lambda
    \in\pg{2\lambda,\lambda},
\end{equation*}
using \eqref{sumI} for the last equality.
However, this gives a contradiction with \eqref{realsum} applied to $\tilde I$, thus proving our claim that $k\in\pg{0,1,n-2,n-1}$.

We will now analyze each one of these cases.
Note that the cases $k=0$ and $k=n-1$ are equivalent, as in both cases all the $u_j$'s have the same value.
In this case, it only remains to compute the possible values of the $u_j$'s.
We observe that \eqref{realsum} does not depend on the choice of the set $I$ in this case, and it is equivalent to $(\lambda+pu_2)pu_2=0$, namely 
\begin{equation*}
    u_2\in\pg{0,-\frac{\lambda}{p}}.
\end{equation*}
In other words, we are in case (iia) or (iib) of the statement.

If $k=1$, then by \eqref{kalambda}, $u_3=\ldots=u_n=u_2+\lambda$, and \eqref{realsum} is equivalent to the following system
\begin{equation*}
    \begin{cases}
       (\lambda+pu_n)pu_n=0,\\
       0=(\lambda+u_2+(p-1)u_n)(u_2+(p-1)u_n)=pu_n(pu_n-\lambda)
    \end{cases}
\end{equation*}
depending on whether $2\notin I$ or $2\in I$.
The only solution to this system is $u_n=0$, and then $u_2=-\lambda$.
Up to rearranging the $u_j$'s, this corresponds to case (iic) in the statement.

It remains to consider the case $k=n-2$.
By \eqref{kalambda} we have $u_2=\ldots=u_{n-1}=u_n-\lambda$, so  \eqref{realsum} is equivalent to 
\begin{equation*}
    \begin{cases}
       (\lambda+pu_2)pu_2=0,\\
       0=(\lambda+(p-1)u_2+u_n)((p-1)u_2+u_n)=(pu_2+2\lambda)(pu_2+\lambda),
    \end{cases}
\end{equation*}
depending on whether $n\notin I$ or $n\in I$. The unique solution is 
\begin{equation*}
    u_2=-\frac{\lambda}{p},\quad\quad u_n=u_2+\lambda=\frac{p-1}{p}\lambda,
\end{equation*}
so we fall in case (iid) of the statement. The converse is obvious in all 4 cases.
\end{proof}

\bibliographystyle{plain}
\bibliography{references}
\end{document}